\documentclass[11pt,oneside, reqno]{amsart}

\usepackage[colorlinks,linkcolor=blue,citecolor=blue, urlcolor=blue]{hyperref}
\usepackage{graphics}
\usepackage{enumitem}
\usepackage{marginnote}	
\usepackage{latexsym, amssymb, amsmath, amsthm}
\usepackage{mathrsfs}
\usepackage{tikz}
\usepackage{tikz-cd}
\usepackage[only,llbracket,rrbracket]{stmaryrd}
\usetikzlibrary{decorations.pathmorphing}

\tikzset{
	squiggly/.style={decorate, decoration={snake, amplitude=1mm, segment length=6mm}},
}

\makeatletter
\@namedef{subjclassname@2020}{%
	\textup{2020} Mathematics Subject Classification}
\makeatother

\newtheorem{theorem}{Theorem}[section]
\newtheorem{lemma}[theorem]{Lemma}
\newtheorem{proposition}[theorem]{Proposition}
\newtheorem{corollary}[theorem]{Corollary}

\theoremstyle{definition}\newtheorem{definition}[theorem]{Definition}
\theoremstyle{definition}
\theoremstyle{definition}\newtheorem{remark}[theorem]{Remark}

\def\ga{\mathfrak{a}}
\def\gb{\mathfrak{b}}

\def\gm{\mathfrak{m}}
\def\gp{\mathfrak{p}}
\def\gq{\mathfrak{q}}

\def\gA{\mathfrak{A}}
\def\gB{\mathfrak{B}}

\def \CA{\mathcal{A}}
\def \CB{\mathcal{B}}
\def \CC{\mathcal{C}}
\def \CD{\mathcal{D}}

\def \CJ{\mathcal{J}}

\def \CR{\mathcal{R}}
\def \CS{\mathcal{S}}

\def \CY{\mathcal{Y}}

\def \CZ{\mathcal{Z}}

\def \Prim{{\rm Prim}}
\def\Max{{\rm Max}}

\def \Spec{{\rm Spec}}
\def \ann{{\rm ann}}
\def \Frac{{\rm Frac}}

\def \mK{\Bbbk}
\def \Z{\mathbb{Z}}
\def \N{\mathbb{N}}

\def \D{\Delta}
\def \d{\delta}
\def \s{\sigma }
\def \l{\lambda}

\begin{document}

	\author{Tao Lu}
	
	
	
	\address{School of Mathematical Science, Yangzhou University, Yangzhou, 225002, China}
	
	\email{taolu@yzu.edu.cn}
	
	\begin{abstract} 
We study a family of liftings $\mathfrak U(\lambda)$ of the super Jordan plane over the infinite cyclic group in characteristic zero. For a nonzero lifting parameter $\lambda$, we identify a natural subalgebra $\mathfrak B$ with the bosonization of $U(\mathfrak{osp}(1|2))$, also known as $sl_{-1}(2)$, and show that the lifting $\mathfrak U(\lambda)$ is obtained from $\mathfrak B$ by an Ore localization at a single distinguished element. This structural description provides a unified approach to the representation theory and ideal structure of $\mathfrak U(\lambda)$. We classify the finite-dimensional simple modules of $\mathfrak B$ and prove that every finite-dimensional $\mathfrak B$-module is completely reducible; the localization then yields the corresponding results for $\mathfrak U(\lambda)$. We determine the prime, primitive, and completely prime ideals of both algebras. Finally, we compute their classical quotient rings, showing in particular that the classical quotient ring of $\mathfrak U(\lambda)$ is a $2\times2$ matrix algebra over the skew field of fractions of the tensor product of its centre and the first Weyl algebra.
	\end{abstract}
	
	\title{From bosonized $\mathfrak{osp}(1|2)$ to liftings of the super Jordan plane}
	\maketitle
	

\section{Introduction}

The lifting method of Andruskiewitsch and Schneider \cite{Andruskiewitsch-Schneider-1,Andruskiewitsch-Schneider-2} is a fundamental tool in the study of pointed Hopf algebras. Given the bosonization of a Nichols algebra, the lifting problem asks for all Hopf algebras whose associated graded Hopf algebra with respect to the coradical filtration is that bosonization. Such Hopf algebras, called \emph{liftings}, may be viewed as deformations of the algebra structure that preserve the underlying coalgebra structure. The method has led to many classification results for pointed Hopf algebras;  see, for instance, \cite{Andruskiewitsch-Angiono-Garcia Iglesias,Andruskiewitsch-Schneider-3,Angiono-Garcia Iglesias} and the references therein. 


The Jordan plane and the super Jordan plane are among the basic finite
Gelfand--Kirillov dimensional Nichols algebras of block type and, in
particular, provide important examples beyond the diagonal setting \cite{Andruskiewitsch-Angiono-Heckenberger}. Their liftings over nilpotent-by-finite groups were determined by Andruskiewitsch, Angiono, and Heckenberger \cite{Andruskiewitsch-Angiono-Heckenberger-1,Andruskiewitsch-Angiono-Heckenberger-2}. A family of liftings of the Jordan
plane over the infinite cyclic group was studied in
\cite{LuJordan2026}, where the centres, automorphism groups, prime and
primitive spectra, and finite-dimensional simple modules were determined.
The family $\mathfrak U(\lambda)$ considered in the present paper is the
corresponding family of liftings of the super Jordan plane over the infinite
cyclic group. We study its algebraic and representation-theoretic structure in characteristic zero.

The super Jordan plane and its bosonization have  been studied from several complementary viewpoints. Finite-dimensional simple and small indecomposable modules were classified in \cite{Andruskiewitsch-Bagio-Della Flora-Flores-1}; Hochschild homology and cohomology were computed in \cite{Reca-Solotar}; and a suitable bosonization and its representation theory were investigated in \cite{Andruskiewitsch-Bagio-Della Flora-Flores-2,LuBosoJordan}. More recently, super-primality, automorphisms, the prime spectrum, and the classical ring of quotients were studied in \cite{Andruskiewitsch-Dumas,Lu2026}. These results reveal a rich algebraic
structure associated with the super Jordan plane and its bosonizations.
For the liftings $\mathfrak U(\lambda)$ over the infinite cyclic group,
however, a structural description connecting their representation theory,
ideal structure, and classical quotient rings has not yet been given. The
purpose of this paper is to provide such a description.



There is a significant distinction between the cases $\lambda=0$ and
$\lambda\neq0$. When $\lambda=0$, the Hopf algebra $\mathfrak U(0)$ is
isomorphic to the bosonization of the super Jordan plane. Its prime
spectrum, classical quotient ring, and simple modules were studied in
\cite{LuBosoJordan} over a field of characteristic different from $2$. The approach relies on realizing a suitable localization of $\mathfrak{U}(0)$ as a matrix algebra.
For $\lambda\neq0$, the deformation leads to a substantially different
algebraic structure. In particular, it gives rise to a distinguished
subalgebra of $\mathfrak U(\lambda)$ isomorphic to the bosonization of
$U(\mathfrak{osp}(1|2))$. Moreover, $\mathfrak U(\lambda)$ can be recovered
from this subalgebra by an Ore localization at a single distinguished
element. This structural relationship is the main feature that allows us
to study the nonzero-parameter case systematically.

In this paper, we assume throughout that $\lambda\neq0$ and $\operatorname{char}\Bbbk=0$, and write $\mathfrak U=\mathfrak U(\lambda)$. We construct a natural subalgebra
$\mathfrak B$ of $\mathfrak U$, generated by suitable combinations of the
standard generators of $\mathfrak U$. Proposition~\ref{21Aug26} identifies
this algebra with the bosonization of the universal enveloping algebra of
$\mathfrak{osp}(1|2)$:
\[
\mathfrak{B}\cong \mathfrak{B}(\mathfrak{osp}(1|2))
:=U(\mathfrak{osp}(1|2))\, \#\, \Bbbk\mathbb Z_2.
\]
We then prove in Proposition~\ref{b23Aug26} that $\mathfrak U$ is obtained
from $\mathfrak B$ by Ore localization at a single distinguished element.
Thus, the lifting $\mathfrak U$ can be studied through the simpler algebra
$\mathfrak B$ together with one localization step.  This structural
relationship provides the main framework for the results of the paper.

The algebra $\mathfrak B(\mathfrak{osp}(1|2))$ is also known as
$sl_{-1}(2)$. It was introduced by Tsujimoto, Vinet, and Zhedanov as the
$q\to-1$ limit of the quantum algebra $sl_q(2)$ and interpreted as a
parabosonic Hopf algebra \cite{TsujimotoVinetZhedanov2011}. Its
Clebsch--Gordan and Racah problems are closely related to the dual $-1$
Hahn and Bannai--Ito polynomials
\cite{GenestVinetZhedanov2013,GenestVinetZhedanov2014}. Further connections
between $\mathfrak{B}(\mathfrak{osp}(1|2))$, the Bannai--Ito algebra, and polynomial realizations
have been developed in \cite{BaseilhacEtAl2018,BergeronVinet}. Thus, the subalgebra
$\mathfrak B$ occurring in our lifting problem is itself an algebra with a substantial representation-theoretic and special-function background.

Our results concern three aspects of the structure of
$\mathfrak B$ and $\mathfrak U$.

 First, we determine their finite-dimensional representation theory. We
 prove that the centre of $\mathfrak B$ is a polynomial algebra in one
 variable and classify all finite-dimensional simple $\mathfrak B$-modules.
 In particular, every finite-dimensional simple $\mathfrak B$-module has
 odd dimension. We further prove that every finite-dimensional
 $\mathfrak B$-module is completely reducible. The localization description then shows that restriction from $\mathfrak U$ to $\mathfrak B$ induces an equivalence between their finite-dimensional module categories. Consequently, the classification and complete reducibility results extend to finite-dimensional $\mathfrak U$-modules.

Second, we determine the prime spectra of $\mathfrak B$ and $\mathfrak U$.
Using a generalized Weyl algebra presentation of $\mathfrak B$, we
determine its prime ideals explicitly. The localization description of
$\mathfrak U$ then allows us to obtain the corresponding description for
$\mathfrak U$. In particular, every nonzero prime ideal of $\mathfrak U$
is primitive, and we determine precisely the completely prime ideals.

Third, we determine the classical quotient rings. We show that, after
localizing a natural subalgebra of $\mathfrak U$ at the powers of one of its
generators, the resulting algebra is a matrix algebra of size two over a
skew polynomial ring. After a suitable change of variables, this skew
polynomial ring is identified with a localization of the first Weyl algebra.
Consequently,
\begin{equation*}
Q(\mathfrak U)
\cong
M_2(\mathbf K),
\end{equation*}
where $\mathbf K$ is the skew field of fractions of
$\Bbbk[C]\otimes A_1$, $C$ is the distinguished central element of
$\mathfrak U$, and $A_1$ denotes the first Weyl algebra. The corresponding
description of $Q(\mathfrak B)$ follows from the fact that $\mathfrak U$ is
an Ore localization of $\mathfrak B$.

A complementary structural result concerns the subalgebra $\mathfrak A$ generated by $g^{\pm1}$, $x$, and $s$. We prove that $\mathfrak A$ is a Hopf subalgebra of $\mathfrak U$ and that $\mathfrak U=\mathfrak A[y;\sigma,\delta]$
is a Hopf Ore extension of $\mathfrak A$. We also determine the classical quotient ring and prime spectrum of $\mathfrak A$. The subalgebra $T$ generated by $g^{\pm1}$ and $x$ belongs to a family of GK-dimension one Hopf algebras studied by Liu \cite{LiuGongXiang}, while $\mathfrak A$ is obtained from $T$ by adjoining $s$ via a nontrivial derivation. This structure provides a useful intermediate algebra between $T$ and the lifting $\mathfrak U$.




The paper is organized as follows. Section~\ref{Preliminaries} collects the basic definitions, notation, and the Hopf Ore extension structure of $\mathfrak U$. Section~\ref{Bosp}  constructs the hidden copy of $\mathfrak{B}(\mathfrak{osp}(1|2))$ inside
$\mathfrak U$ and proves that $\mathfrak U$ is a localization of this
subalgebra. Section~\ref{Bspec} studies finite-dimensional modules and prime ideals of $\mathfrak B$. Section~\ref{Uspec} obtains the corresponding results for $\mathfrak U$. Section~\ref{Quotient} determines the classical quotient ring of $\mathfrak U$, studies the auxiliary subalgebras $T$ and $\mathfrak A$, and deduces in particular the classical quotient ring of $\mathfrak B$.

Throughout the paper, all modules are left modules, and the base field
$\Bbbk$ is assumed to be algebraically closed of characteristic zero. We write $\mathbb Z$ for
the ring of integers, $\mathbb N$ for the set of nonnegative integers,
$\mathbb Z^{\mathrm{odd}}$ for the set of odd integers, and
$\mathbb N^{\mathrm{odd}}$ for the set of odd positive integers.
The centre of a ring $R$ is denoted by $Z(R)$. For a subset $\CS \subset R$, we write $\langle \CS\rangle$ for the two-sided ideal generated by $\CS$. 
 The notation $[A,B]=AB-BA$ and $\{A,B\}=AB+BA$ will be used.

\section{Preliminaries} \label{Preliminaries}  
In this section we collect the background material used in the paper.
We first recall the bosonization of the super Jordan plane and its liftings
over the infinite cyclic group. We then
describe a Hopf Ore extension presentation of
$\mathfrak{U}=\mathfrak{U}(\lambda)$ and the associated chain of Hopf
subalgebras. Finally, we recall the bosonization of
$U(\mathfrak{osp}(1|2))$, which will serve as the comparison algebra in
Section~3.

\subsection{Bosonization of the super Jordan plane}
We begin by recalling the super Jordan plane and its bosonization.
 The \emph{super Jordan plane} $\CJ$ is the $\mK$-algebra generated by $x$  and $y$ subject to the defining relations
\begin{equation*} 
x^2=0, \qquad y s-sy-xs=0, \quad {\rm where}\,\, s:=xy+yx. 
\end{equation*}

The super Jordan plane is a Nichols algebra first introduced in \cite{Andruskiewitsch-Angiono-Heckenberger}. It has Gelfand--Kirillov dimension 2, and the elements $x^i s^j y^k$ ($i=0,1$; $j,k\in \N$) form a PBW basis of $\CJ$; see \cite{Andruskiewitsch-Angiono-Heckenberger}. Furthermore, $\CJ$ is Noetherian of infinite global dimension by \cite{Reca-Solotar}, and it admits a presentation as an Ore extension by \cite[Proposition 1.3]{Andruskiewitsch-Dumas}.

Let $G=\langle g\rangle$ be the infinite cyclic group with generator $g$. Consider the Yetter--Drinfeld module $V=\mK x\oplus \mK y$ in $_{\mK G}^{\mK G}\CY\CD$ with grading $V_g=V$ and action 
\[g\cdot x=-x, \qquad g\cdot y=-y+x.\]
 Then $\CJ$ is the Nichols algebra $\CB(V)$; see \cite[Proposition 3.5]{Andruskiewitsch-Angiono-Heckenberger}. Thus, $\CJ$
 is a Hopf algebra in the braided category
 ${}_{\mK G}^{\mK G}\CY\CD$, and its Radford--Majid bosonization
 \[
 A:=\CJ\#\mK G
 \]
 is a Hopf algebra in the category of $\mK$-vector spaces. We recall its
 explicit algebra and coalgebra structures below.

\begin{definition}
	The \emph{bosonization of the super Jordan plane} is the $\mK$-algebra $A$ generated by $x,y,g,g^{-1}$ subject to the relations
	\begin{equation*}
		\begin{gathered} 
		gg^{-1}=g^{-1}g=1, \quad gxg^{-1}=-x, \quad gyg^{-1}=-y+x, \\
	x^2=0, \\
	y s-sy-xs=0, \quad {\rm where}\,\, s:=xy+yx.\\ 
	\end{gathered}
	\end{equation*}
	Moreover, $A$ admits a Hopf algebra structure with comultiplication given by
	\begin{equation*}
	\D(g^{\pm 1})=g^{\pm 1} \otimes g^{\pm 1}, \quad \D(x)=x\otimes 1+ g\otimes x, \quad \D(y)=y \otimes 1+ g\otimes y. 
	\end{equation*}
\end{definition}

The algebra $A$ is a pointed Hopf algebra of Gelfand--Kirillov dimension $3$.
Its prime spectrum, classical
quotient ring, and simple modules were studied in
\cite{LuBosoJordan}, and its Drinfeld double
was studied in \cite{Andruskiewitsch-Pena Pollastri}.

\subsection{Liftings of the super Jordan plane}
We now recall the liftings of the super Jordan plane over the infinite cyclic group from \cite{Andruskiewitsch-Angiono-Heckenberger-1}. 
\begin{definition}
	Let $\l \in \mK$. Define $\mathfrak{U}(\l)$ to be the associative $\mK$-algebra generated by $g^{\pm 1}$, $x$, and $y$ subject to the relations
	\begin{equation} \label{DefU} 
	\begin{gathered}
	gg^{-1}=g^{-1}g=1, \quad gxg^{-1}=-x, \quad gyg^{-1}=-y+x,\\
	x^2=\l(1-g^2), \\
	ys-sy=xs-2\l y -\l xg^2, \quad \text{where } s:=xy+yx. 
	\end{gathered}
	\end{equation}
	Moreover, $\mathfrak{U}(\lambda)$ is a Hopf algebra with comultiplication, counit, and antipode given by
	\begin{equation*}
	\begin{gathered}
		\D(g^{\pm 1})=g^{\pm 1} \otimes g^{\pm 1}, \quad \D(x)=x\otimes 1+ g\otimes x, \quad \D(y)=y \otimes 1+ g\otimes y, \\
		\varepsilon(g)=1, \quad \varepsilon(x)=0, \quad \varepsilon(y)=0, \\
		S(g)=g^{-1}, \quad S(x)=-g^{-1}x, \quad S(y)=-g^{-1}y. 
	\end{gathered}
	\end{equation*}
\end{definition}
The defining relations \eqref{DefU} imply
\begin{equation} 
[x,s]=2\l x g^2, \qquad gsg^{-1}=s-2x^2.  \label{xsg} 
\end{equation}

By \cite[Proposition 4.4]{Andruskiewitsch-Angiono-Heckenberger-1}, for every $\lambda\in\mK$, the algebra $\mathfrak{U}(\lambda)$ is a pointed Hopf algebra of Gelfand--Kirillov dimension $3$, and every lifting of the super Jordan plane over the infinite cyclic group is isomorphic to $\mathfrak{U}(\lambda)$ for some $\lambda\in\mK$. The associated graded Hopf algebra of $\mathfrak{U}(\lambda)$ with respect to the coradical filtration is
\begin{equation*}
\operatorname{gr}\mathfrak{U}(\lambda)\cong A,
\end{equation*}
and $\mathfrak{U}(\lambda)$ is a cocycle deformation of $A$. Moreover, $\mathfrak{U}(\l) \cong \mathfrak{U}(\l')$ if and only if there exists $c\in \mK^*$ such that $\l=c\l'$.

When $\lambda=0$, the defining relations reduce to those of the bosonization $A$, and hence
\[\mathfrak{U}(0)\cong A.\]

Throughout the paper we assume $\lambda\in\mK^*$ and write $\mathfrak{U}:=\mathfrak{U}(\l)$.

\subsection{$\mathfrak{U}$ as a Hopf Ore extension}
For $q\in\mK^*$, let 
\begin{equation*}
\mK_q[a,b]:=\mK\langle a, b \mid ab=qba \rangle
\end{equation*}
 denote the quantum plane. If $q$ is a primitive $\ell$-th root of unity, then $Z(\mK_q[a,b])=\mK[a^\ell,b^\ell].$

Let $T$ be the subalgebra of $\mathfrak{U}$ generated by $g^{\pm 1}$ and $x$.
The formulas for $\Delta$, $\varepsilon$, and $S$ on the generators $g^{\pm1}$
and $x$ show that $T$ is a Hopf subalgebra of
$\mathfrak{U}$, and
\begin{equation} \label{Tgx}  
T\cong\mK_{-1}[g^{\pm 1}, x]/\langle x^2-\l(1-g^2) \rangle. 
\end{equation}
In particular, $T=\mK[g^{\pm 1}] \oplus \mK[g^{\pm 1}]x$ as a vector space. Its centre is $Z(T)=\mK[g^{\pm 2}]$, and $T$ is a free $Z(T)$-module of rank 4 with
$T=\oplus_{i,j=0}^1Z(T) g^i x^j.$ Consequently, 
\[\operatorname{GKdim} T=\operatorname{GKdim}Z(T)=1.\] 
This Hopf subalgebra is a special case of the family of prime regular Hopf algebras of
Gelfand--Kirillov dimension one introduced by Liu
\cite[Proposition 2.1]{LiuGongXiang}.

The algebra $\mathfrak{U}$ can be presented as an iterated Ore extension over $T$ in the form
\begin{equation} \label{Uite}  
\mathfrak{U}=T[s;\d_1][y;\s_2, \d_2],
\end{equation}
where $\d_1$ is the derivation of $T$ determined by 
\begin{equation*}
\d_1(g)=2x^2 g, \quad \d_1(x)=-2\l g^2 x, 
\end{equation*}
and $\sigma_2$ and $\delta_2$ are respectively the automorphism and $\sigma_2$-derivation of $T[s;\delta_1]$ determined by
\begin{equation*}
\begin{aligned}
\s_2(g)&=-g,  &\s_2(x)&=-x, &\s_2(s)&=s-2\l\\
\d_2(g)&=xg,  &\d_2(x)&=s, &\d_2(s)&=xs-\l xg^2. 
\end{aligned}
\end{equation*}
It follows that $\mathfrak{U}$ is Noetherian and
$\operatorname{GKdim}\mathfrak{U}=3$. Moreover,
\begin{equation*}
\left\{
g^i x^j s^k y^\ell
\mid
i\in\mathbb Z,\ j\in\{0,1\},\ k,\ell\in\mathbb N
\right\}
\end{equation*}
is a $\mK$-basis of $\mathfrak{U}$.


\begin{definition}
	Let $R$ be a Hopf $\mK$-algebra. A \emph{Hopf Ore extension} of $R$ is an algebra $H$ satisfying the following conditions:
	\begin{enumerate}[label=(\roman*)]
		\item $H$ is a Hopf algebra containing $R$ as a Hopf subalgebra;
		\item there exist an algebra automorphism $\sigma$ of $R$ and a $\sigma$-derivation $\delta$ of $R$ such that $H=R[x;\sigma,\delta]$.
	\end{enumerate}
\end{definition}

Let $\mathfrak{A}:=T[s;\d_1] \subset \mathfrak{U}$. Since $s=xy+yx$, a direct computation gives
\begin{equation*}
\D(s)=s\otimes 1 + xg \otimes x + g^2 \otimes s, \quad \varepsilon(s)=0, \quad S(s)=-g^{-2}(s-x^2). 
\end{equation*} 
Hence $\mathfrak{A}$ is a Hopf subalgebra of $\mathfrak{U}$.  In particular, \eqref{Uite} exhibits $\mathfrak{U}$ as a two-step iterated Hopf Ore extension of $T$.
We will consider the chain of Hopf subalgebras 
\[T \subset \mathfrak{A} \subset \mathfrak{U}\]
 as well as their localizations at the powers of $x$ and their classical quotient rings.

\subsection{The bosonization of $U(\mathfrak{osp}(1|2))$}
We end this section by recalling the bosonization of
$U(\mathfrak{osp}(1|2))$. Its algebra structure will be used in the next
section, while the Hopf structure is recorded only for completeness.

The universal enveloping algebra $U(\mathfrak{osp}(1|2))$ is generated, as an associative algebra, by the elements $H$, $X$, and $Y$ subject to the  relations
\begin{equation}  \label{Uosp12}  
HX-XH=X, \qquad HY-YH=-Y, \qquad XY+YX=H. 
\end{equation}
The elements $\{ Y^i H^j X^k \mid i,j,k\in \N \}$ form a PBW basis of $U(\mathfrak{osp}(1|2))$. The ring-theoretic properties of $U(\mathfrak{osp}(1|2))$ were studied by Pinczon \cite{Pinczon1990}.

\begin{definition}
	The \emph{bosonization $\mathfrak{B}(\mathfrak{osp}(1|2))$ of the enveloping algebra $U(\mathfrak{osp}(1|2))$} is the associative $\mK$-algebra generated by $H, X, Y$, and $P$ subject to the relations \eqref{Uosp12} together with
	\begin{equation*}
	P^2=1, \quad PH=HP, \quad PX=-XP, \quad PY=-YP. 
	\end{equation*}
	For completeness,  it is a Hopf algebra with comultiplication, counit, and antipode given by
	\begin{equation*}
	\begin{gathered}
	\D(H)=H\otimes 1+1\otimes H, \quad \D(X)=X\otimes 1 + P\otimes X, \\
	 \D(Y)=Y  \otimes 1+ P\otimes Y, \quad \D(P)=P \otimes P, \\
	\varepsilon(H)=\varepsilon(X)=\varepsilon(Y)=0, \quad \varepsilon(P)=1, \\
	S(H)=-H, \quad S(X)=-PX, \quad S(Y)=-PY, \quad S(P)=P^{-1}. 
	\end{gathered}
	\end{equation*}
		We will not use this Hopf structure in the sequel.
\end{definition}
The elements $ \{Y^iH^jX^kP^\epsilon
\mid i,j,k\in\mathbb{N},\ \epsilon\in\{0,1\}\}$ form a PBW basis of $\gB(\mathfrak{osp}(1|2))$. 


\section{The hidden $\mathfrak{B}(\mathfrak{osp}(1|2))$}  \label{Bosp} 

In this section, we construct a distinguished subalgebra
$\mathfrak B\subseteq\mathfrak U$ and prove that it is isomorphic to the
bosonization $\mathfrak B(\mathfrak{osp}(1|2))$. We then show that
$\mathfrak U$ is obtained from $\mathfrak B$ by localizing at a single
explicit element. This provides the main structural connection between
$\mathfrak U$ and $\mathfrak B$, and will allow us to transfer representation-theoretic and ring-theoretic information between the two algebras in the subsequent sections.

\begin{lemma} \label{DefB} 
	Define elements of $\mathfrak U$ by
	\begin{equation*}
	P:= g+\frac{1}{\sqrt{\l}}x, \quad H:=-\frac{1}{2\sqrt{\l}}\{y,P\}, \quad X:=xg^{-1}, \quad Y:=\frac{1}{4\l}[y, P]. 
	\end{equation*}
	Then these elements satisfy
	\begin{equation}
\begin{gathered} \label{XYP} 
P^2=1, \quad PH=HP,   \quad PX=-XP, \quad PY=-YP, \\
	[H, X]= X, \qquad [H, Y]=-Y, \qquad XY+YX=H.  
\end{gathered}	
\end{equation}
\end{lemma}
\begin{proof}
		We first verify the relations involving $P$. Using $x^2=\lambda(1-g^2)$ and $gx=-xg$, we obtain
		\begin{equation*}
		P^2
		=g^2+\frac1{\sqrt{\lambda}}(gx+xg)+\frac{x^2}{\lambda}
		=g^2+1-g^2=1.
		\end{equation*}
	Moreover, $P(yP+Py)=PyP+y=(yP+Py)P$, so $P$ commutes with $\{y,P\}$, and hence $PH=HP$. Since $gx=-xg$, we immediately obtain $PX=-XP$. Finally, \[P(yP-Py)+(yP-Py)P=PyP-y+y-PyP=0,\]
	 so $P$ anticommutes with $[y,P]$, and consequently $PY=-YP$. 
	
	We next compute the commutators with $H$. First, $$\{y,P\}=yg+gy+\frac1{\sqrt{\lambda}}(yx+xy).$$ Since $yg+gy=xg$ and $xy+yx=s$, it follows that
	\begin{equation} \label{HH}  
H=-\frac{1}{2\l}(s+\sqrt{\l}\,xg).
	\end{equation}
	Using \eqref{xsg}, we obtain $[s, xg^{-1}]=-2\l xg^{-1}$. Observe that $xg$ commutes with $xg^{-1}$. Thus
	\begin{equation*}
	[H,X]=-\frac{1}{2\l}[s+\sqrt{\l}\, xg, xg^{-1}]=-\frac{1}{2\l}[s, xg^{-1}]=X. 
	\end{equation*}

Next, we rewrite $Y$ as
\begin{equation} \label{YY} 
Y=\frac{1}{4\l}(2yP-\{y,P\})=\frac{1}{2\l}yP+\frac{1}{2\sqrt{\l}}H. 
\end{equation}
Since $H$ commutes with $P$, we have $[H,Y]=\frac{1}{2\l}[H,y]P$. We now compute $[H,y]$.  Using $[y, xg]=gs+x^2g$ and the defining relations of $\mathfrak U$, we obtain
\begin{equation*}
\begin{aligned}\
[H,y]&=-\frac{1}{2\l}[s+\sqrt{\l}\,xg, \,y]\\
&=-\frac{1}{2\l}\left(  -xs+2\l y+ \l xg^2-\sqrt{\l} gs -\sqrt{\l} x^2 g   \right)\\
&=-\frac{1}{2\l}\left(  -(x+\sqrt{\l}g) s+2\l y -\sqrt{\l} (x+\sqrt{\l}g)xg   \right)\\
&=-\frac{1}{2\l}\left(  2\l y -(x+\sqrt{\l}g)(s+\sqrt{\l}xg)  \right)\\
&= -\frac{1}{2\l}\left(  2\l y -\sqrt{\l}P(-2\l H)  \right)\\
&=-\left(y+\sqrt{\l} PH \right). 
\end{aligned}
\end{equation*}
Therefore, 
\[[H,Y]=\frac{1}{2\l}[H,y]P=-(\frac{1}{2\l}yP+\frac{1}{2\sqrt{\l}}H)=-Y.\] 

It remains to prove that $XY+YX=H$. Using the expression for $Y$ in \eqref{YY}, we get
\begin{equation*}
\begin{aligned}
XY+YX=\frac{1}{2\l}X\left(yP+\sqrt{\l}H \right)+\frac{1}{2\l}\left(yP+\sqrt{\l}H\right)X.
\end{aligned}
\end{equation*}
Since $\{ y,P \}=-2\sqrt{\l}H$, we have $yP+\sqrt{\l}H=-Py-\sqrt{\l}H$. Hence
\begin{equation*}
\begin{aligned}
XY+YX &=\frac{1}{2\l}X\left(-Py-\sqrt{\l}H\right)+\frac{1}{2\l}\left(yP+\sqrt{\l}H\right)X\\
&=\frac{1}{2\l} \left(  yPX-XPy+\sqrt{\l}X\right)\\
&=\frac{1}{2\l} \left(  yPX+PXy+\sqrt{\l}X\right).
\end{aligned}
\end{equation*}
Now,
\begin{equation*}
PX
=\bigl(g+\frac{x}{\sqrt{\lambda}}\bigr)xg^{-1}
=-x+\sqrt{\lambda}(g^{-1}-g)
=\sqrt{\lambda}(g^{-1}-P).
\end{equation*}
Furthermore, $\{y, g^{-1}\}=g^{-1}x=-X$, while $\{ y,P \}=-2\sqrt{\l}H$. Consequently,
\begin{equation*}
\{y, PX \}=\sqrt{\l}\,\bigl(\{y,g^{-1}\}-\{y,P\}\bigr)=-\sqrt{\l}X+2\l H. 
\end{equation*}
Therefore, 
\[XY+YX =\frac{1}{2\l} \bigl(  \{y, PX \}+\sqrt{\l}X\bigr) =H.\] This proves all the relations in \eqref{XYP}.
\end{proof}

Let $\gB$ be the subalgebra of $\mathfrak{U}$ generated by $P$, $H$, $X$, and $Y$.
By Lemma~\ref{DefB}, these generators satisfy the defining
relations of $\mathfrak B(\mathfrak{osp}(1|2))$.
 
 \begin{proposition} \label{21Aug26}  
 	There is an algebra isomorphism $\mathfrak{B} \cong \mathfrak{B}(\mathfrak{osp}(1|2))$. 
 \end{proposition}
\begin{proof}
	By Lemma~\ref{DefB}, there is a surjective homomorphism
	\[
	\pi\colon \mathfrak B(\mathfrak{osp}(1|2))\twoheadrightarrow \mathfrak B,
	\qquad
	P\mapsto P,\ H\mapsto H,\ X\mapsto X,\ Y\mapsto Y.
	\]
	We prove that $\pi$ is injective.

	Let $\mathcal V:=\Bbbk\langle H,X,Y\rangle\subseteq
	\mathfrak B(\mathfrak{osp}(1|2))$.
	Then $\mathcal V\cong U(\mathfrak{osp}(1|2))$ is a
	Noetherian domain. Let $Q:=\Frac(\mathcal V)$ be the skew field of fractions of $\mathcal V$, and let
	$\alpha\in \operatorname{Aut}(\mathcal V)$ be the involution determined by
	\[
	\alpha(H)=H,\qquad \alpha(X)=-X,\qquad \alpha(Y)=-Y.
	\]
	Since
	$\mathfrak B(\mathfrak{osp}(1|2))=\mathcal V\# \langle P\rangle=\mathcal V \oplus \mathcal{V}P$,
	there is an algebra morphism
	\[
	\theta\colon \mathfrak B(\mathfrak{osp}(1|2))\longrightarrow M_2(Q)
	\]
	given by
	\[
	\theta(v)=
	\begin{pmatrix}
	v&0\\
	0&\alpha(v)
	\end{pmatrix}
	\qquad (v\in \mathcal V),\qquad
	\theta(P)=
	\begin{pmatrix}
	0&1\\
	1&0
	\end{pmatrix}.
	\]
  The crossed-product relation is encoded by
$\theta(P)\theta(v)\theta(P)=\theta(\alpha(v))$. Moreover,
\[
\theta(v+wP)=
\begin{pmatrix}
v&w\\
\alpha(w)&\alpha(v)
\end{pmatrix}
\qquad (v,w\in\mathcal V),
\]
so $\theta(v+wP)=0$ forces $v=w=0$. Thus $\theta$ is injective.
	
	Set
	$t:=1+\sqrt{\lambda}^{-1}X\in \mathcal V$.
	Since $\mathcal V$ is a domain, both $t$ and
	$\alpha(t)=1-\sqrt{\lambda}^{-1}X$
	are nonzero, hence invertible in $Q$. Therefore
	$\theta(t)$
	is invertible in $M_2(Q)$. Define
	\[
	\widetilde g:=\theta(t)^{-1}\theta(P),\qquad
	\widetilde x:=\theta(X)\widetilde g,\qquad
	\widetilde y:=\theta((2\lambda Y-\sqrt{\lambda}\,H)P).
	\]
	Since $PX=-XP$, we have
	\[
	\widetilde g^{-1}=\theta(P)\theta(t),
	\qquad
	\widetilde g\,\theta(X)\,\widetilde g^{-1}=-\theta(X).
	\]
	Hence
	\[
	\widetilde x\widetilde g^{-1}=\theta(X),\qquad
	\widetilde g\widetilde x\widetilde g^{-1}=-\widetilde x,
	\]
	and also
	\[
	\widetilde g+\frac{1}{\sqrt{\lambda}}\widetilde x
	=\theta\!\left(1+\sqrt{\lambda}^{-1}\,X\right)\widetilde g
	=\theta(P).
	\]
	Using $PH=HP$ and $PY=-YP$, we obtain
	\[
	-\frac{1}{2\sqrt{\lambda}}\{\widetilde y,\theta(P)\}=\theta(H),\qquad
	\frac{1}{4\lambda}[\widetilde y,\theta(P)]=\theta(Y).
	\]
	Hence 
	$\theta(P),\theta(H),\theta(X),\theta(Y)$ all lie in the subalgebra of $M_2(Q)$ generated by 
	$\widetilde g^{\pm1},\widetilde x,\widetilde y$.
	
   A direct computation shows that $\widetilde g^{\pm1},\widetilde x,\widetilde y$, together with $\widetilde s:=\widetilde x\widetilde y+\widetilde y\widetilde x$ satisfy the relations
	\eqref{DefU}.
	Hence, by the universal property
	of $\mathfrak U$, there is a homomorphism
	\[
	\Psi\colon \mathfrak U\longrightarrow M_2(Q)
	\]
	such that
	\[
	\Psi(g)=\widetilde g,\qquad
	\Psi(x)=\widetilde x,\qquad
	\Psi(y)=\widetilde y.
	\]	
	By construction,
	\[
	\Psi(P)=\theta(P),\qquad
	\Psi(H)=\theta(H),\qquad
	\Psi(X)=\theta(X),\qquad
	\Psi(Y)=\theta(Y).
	\]
	Therefore
	$\Psi|_{\mathfrak B}\circ \pi=\theta$.
	Since $\theta$ is injective, the map $\pi$ is injective as well. Hence $\pi$
	is an isomorphism, and therefore
	$\mathfrak B\cong \mathfrak B(\mathfrak{osp}(1|2))$.
\end{proof}

\begin{remark}
	Although Proposition \ref{21Aug26} gives an algebra isomorphism
	$\mathfrak B\cong \mathfrak B(\mathfrak{osp}(1|2))$, the inclusion
	$\mathfrak B\hookrightarrow \mathfrak U$ is not a morphism of Hopf
	algebras. Indeed, since
	$P=g+\frac{1}{\sqrt{\lambda}}x$,
	the comultiplication in $\mathfrak U$ gives
	\[
	\Delta(P)
	=\Delta(g)+\frac{1}{\sqrt{\lambda}}\Delta(x)
	=P\otimes 1+g\otimes (P-1).
	\]
Therefore $\Delta(P)\notin \mathfrak B\otimes \mathfrak B$, and $\mathfrak B$ is not a Hopf subalgebra of $\mathfrak U$.
\end{remark}

The preceding proposition identifies the distinguished subalgebra of
$\mathfrak U$ with the bosonization of $U(\mathfrak{osp}(1|2))$. We next show
that $\mathfrak U$ is obtained from this subalgebra by inverting one natural
element.

\begin{lemma} \label{23Aug26}  
	Let $w:=\lambda-X^2\in \mathfrak{B}$. Then $w$ is ad-locally nilpotent in $\mathfrak{B}$. 
\end{lemma}
\begin{proof}
	Since $PX=-XP$, we have $wP=Pw$.  Clearly, $w$ also commutes with $X$. Thus
	$\operatorname{ad}_w(P)=\operatorname{ad}_w(X)=0$.
	Moreover, $[w,H]=[H,X^2]=2X^2$, and
	\begin{equation*}
	[w, Y]=[Y,X^2]=YX^2-X^2Y=(H-XY)X-X(H-YX)=HX-XH=X.
	\end{equation*}
	Thus $\operatorname{ad}_w^2(H)=0$, and $\operatorname{ad}_w^2(Y)=0$. Since $\operatorname{ad}_w$ is a derivation, the Leibniz rule now implies that
	$\operatorname{ad}_w$ is locally nilpotent on $\mathfrak B$.
\end{proof}

By Lemma~\ref{23Aug26}, the multiplicative set
$S:=\{w^n\mid n\in\mathbb N\}$
is a left and right Ore set in $\mathfrak B$. Let $\mathfrak{B}_w$ denote the localization of $\mathfrak{B}$ at this Ore set. 
\begin{proposition} \label{b23Aug26}  
	There is an algebra isomorphism $\mathfrak{U} \cong \mathfrak{B}_w$. 
\end{proposition}
\begin{proof}
	Inside $\mathfrak{U}$, one has $X=xg^{-1}$ and $x^2=\lambda(1-g^2)$. Hence
	\begin{equation*}
	X^2=xg^{-1}xg^{-1}=-x^2g^{-2}=-\lambda(1-g^2)g^{-2}=\lambda-\lambda g^{-2},
	\end{equation*}
	so $w=\lambda-X^2=\lambda g^{-2}$. Hence every element of $\{w^n \mid n\in \N \}$ is invertible in $\mathfrak{U}$, and the inclusion $\mathfrak{B} \hookrightarrow \mathfrak{U}$ induces an algebra morphism $\phi: \mathfrak{B}_w \rightarrow \mathfrak{U}$. 
	
	To prove surjectivity, it is enough to show that $g^{\pm 1} \in \mathfrak{B}_w$, because then 
	\begin{equation*}
	x=Xg\in \mathfrak{B}_w, \qquad y= (2\l Y-\sqrt{\l} H)P\in \mathfrak{B} \subseteq \mathfrak{B}_w,
	\end{equation*}
	and $\mathfrak{U}$ is generated by $g^{\pm 1}$, $x$, and $y$. Now
	\begin{equation*}
	(\sqrt{\lambda}+X)(\sqrt{\lambda}-X)=\lambda-X^2=w, 
	\end{equation*}
    and $X$ commutes with $w$. Hence in $\mathfrak{B}_w$, 
    \begin{equation*}
    (\sqrt{\lambda}+X)^{-1}= w^{-1}(\sqrt{\lambda}-X).
    \end{equation*}
    From the equality $(\sqrt{\l}+X)g=\sqrt{\l}P$, we obtain
    \begin{equation*}
    g=\sqrt{\lambda}w^{-1}(\sqrt{\lambda}-X)P\in \mathfrak{B}_w.
    \end{equation*}
    On the other hand, $g^{-1}=\frac{1}{\sqrt{\lambda}} P(\sqrt{\lambda}+X) \in \mathfrak{B} \subseteq \mathfrak{B}_w.$
    Thus $\phi$ is surjective. 
    
    It remains to prove injectivity. Since $\mathfrak B$ is Noetherian and
    $w$ is a regular element, the localization map
    $\mathfrak B\rightarrow\mathfrak B_w$
    is injective. Moreover, the image of $\mathfrak B$ in $\mathfrak U$ is the
    given subalgebra $\mathfrak B$, so the induced map $\phi$ is injective as well.
    Therefore $\mathfrak U \cong \mathfrak{B}_w$. 
\end{proof}

We have therefore obtained the fundamental structural description
\[
\mathfrak B\cong
	\mathfrak B(\mathfrak{osp}(1|2)),
	\qquad
	\mathfrak U(\lambda)\cong\mathfrak B_{\lambda-X^2}.
\]
In particular, the lifting $\mathfrak U(\lambda)$ is recovered from the
bosonized $\mathfrak{osp}(1|2)$ by inverting the single distinguished element
$\lambda-X^2$. This description will be the basis for the study of finite-dimensional modules and prime ideals in the next two sections.

\section{Finite-dimensional modules and prime ideals of  $\mathfrak{B}$}
\label{Bspec} 

In this section, we study the finite-dimensional $\mathfrak B$-modules and the
prime spectrum of $\mathfrak B$ using its realization as a generalized Weyl
algebra. We first determine the centre of $\mathfrak B$, then classify its
finite-dimensional simple modules, and finally describe its prime, primitive,
and completely prime ideals.

\subsection{The centre of $\gB$}
Using the relations \eqref{XYP}, one readily verifies that
\begin{equation*}
C:=-4YXP+2HP+P
\end{equation*}
is central in $\mathfrak B$. Since $XY+YX=H$, we also have
\begin{equation} \label{YXY} 
YX=\frac{1}{4}(2H-CP+1), \qquad XY=\frac{1}{4}(2H+CP-1).
\end{equation}



By Proposition~\ref{b23Aug26}, $\mathfrak U$ is a localization of
$\mathfrak B$. Hence $C$ remains central in $\mathfrak U$. In terms of the
original generators of $\mathfrak U$, it is given by
\begin{equation} \label{Cent}  
C= \frac{1}{\l}\bigl(2yx-s+\l g^2\bigr)g^{-1}=\frac{1}{\l}\bigl(-2xy+s+\l g^2\bigr)g^{-1}. 
\end{equation}
Indeed, by \eqref{YY} and \eqref{HH} we have
\[
y=(2\l Y-\sqrt{\l}\,H)P, \qquad s=-2\l H-\sqrt{\l}\,xg.
\]
Using also $x=Xg$ and $PX=-XP$, we obtain
\[
2yx=2(2\l Y-\sqrt{\l}\,H)PXg=-4\l YXPg+2\sqrt{\l}\,HXPg.
\]
Next,
\[
XPg=xg^{-1}Pg=xg-\frac{x^2}{\sqrt{\l}}
=xg-\sqrt{\l}+\sqrt{\l}\,g^2,
\]
where we used $x^2=\l(1-g^2)$. Therefore
\begin{equation*}
\begin{aligned}
2yx-s+\l g^2
&=-4\l YXPg+2\sqrt{\l}\,HXPg+2\l H+\sqrt{\l}\,xg+\l g^2\\
&=-4\l YXPg+2\sqrt{\l}\,Hxg-2\l H+2\l Hg^2+2\l H+\sqrt{\l}\,xg+\l g^2\\
&=-4\l YXPg+2\l H\left(g^2+\frac{1}{\sqrt{\l}}xg\right)
+\l\left(g^2+\frac{1}{\sqrt{\l}}xg\right)\\
&=\l(-4YXP+2HP+P)g\\
&=\l Cg.
\end{aligned}
\end{equation*}
Multiplying on the right by $g^{-1}$ proves the first equality in
\eqref{Cent}. The second equality follows from the relation $xy+yx=s$.

We now recall the definition of generalized Weyl algebras from \cite{BavulaGWA}, and realize $\mathfrak B$ as a generalized Weyl algebra.

\begin{definition}
	Let $D$ be a ring, let $\s$ be an automorphism of $D$, and let $\mathtt{a} \in Z(D)$. The \emph{generalized Weyl algebra} (GWA, for short) $A:=D[x,y;\s, \mathtt{a}]$ is a ring generated by $D$, $x$, and $y$ subject to the following defining relations:
	\begin{equation*}
	\begin{gathered}
	xd=\s(d)x, \qquad yd=\s^{-1}(d) y \quad \text{for all } d\in D, \\
	yx=\mathtt{a}, \qquad xy=\s(\mathtt{a}). 
	\end{gathered}
	\end{equation*}
\end{definition}
	The generalized Weyl algebra $A$ is $\Z$-graded, with
\begin{equation} \label{DGWA}  
A=\bigoplus_{i\in \Z} A_i, \qquad A_i=\begin{cases}
Dx^i, & \text{ for } i > 0, \\
Dy^{|i|}, & \text{ for } i \leq  0.
\end{cases}
\end{equation}

From \eqref{XYP}  and \eqref{YXY}, the algebra $\gB$ admits a realization as a generalized Weyl algebra. More precisely, let $D:=\mK[C, H, P\mid P^2=1]\cong \mK[C,H,P]/\langle P^2-1\rangle$, then
\begin{equation} \label{UGWA}  
\gB=D[X,Y;\s, \mathtt{a}:=\frac{1}{4}(2H-CP+1)]
\end{equation}
where $\s$ is the automorphism of $D$ determined by $\s(C)=C$, $\s(H)=H-1$, and $\s(P)=-P$. From \eqref{DGWA}, we obtain
\begin{equation} \label{Ugrad} 
\gB=\bigoplus_{i\geq 0}D Y^i \oplus \bigoplus_{i>0} D X^i. 
\end{equation}

\begin{proposition} \label{5July26}  
	The centre of $\gB$ is $Z(\gB)=\mK[C]$. 
\end{proposition}
\begin{proof}
	Let $\gB_X$ denote the localization of $\gB$ at the Ore set $\{X^i\mid i\in \N\}$. By \eqref{UGWA}, $\gB_X$ is a skew Laurent polynomial algebra 
	\begin{equation*}
	\gB_X \cong D[X^{\pm 1};\s],
	\end{equation*}
	where $\s$ is the automorphism of $D$ determined by $\s(C)=C$, $\s(H)=H-1$, and $\s(P)=-P$.
	
	Let $z \in Z(\gB_X)$ and write $z=\sum_{i=m}^n d_i X^i$ with $d_i \in D$. The relation $Hz=zH$ implies $0=Hz-zH=\sum_{i=m}^n id_i X^i$, hence $d_i=0$ for all $i\neq 0$, and thus $z=d_0\in D$.  The condition $Xz=zX$ is equivalent to $\sigma(d_0)=d_0$, so $Z(\gB_X)=D^{\s}$.  
	
	Next, write $D=\mK[C,H]\oplus \mK[C,H]P$. Let $z=f+hP \in D^{\s}$ with $f, h\in \mK[C,H]$. The condition $\sigma(z)=z$ yields $\s(f)=f$ and $\s(h)=-h$. The second identity forces $h=0$, so $z=f\in \mK[C,H]^{\s}$. Since $\sigma(H)=H-1$ is a nontrivial translation, the fixed subalgebra of $\mK[C,H]$ is $\mK[C]$. Therefore, $D^{\sigma}=\mK[C]$.	
	Finally, $Z(\gB)=\gB\cap Z(\gB_X)=\mK[C]$. 
\end{proof}

\subsection{Finite-dimensional $\gB$-modules}
For $\epsilon\in \{\pm 1 \}$ and $\alpha \in \mK$, define
\begin{equation*}
M(\epsilon, \alpha):=\gB/\gB(P-\epsilon, \,C-\alpha,\, X).
\end{equation*}
Since $C=-4YXP+2HP+P$, we have $H-\frac{1}{2}(\epsilon \alpha-1)\in \gB(P-\epsilon, C-\alpha, X).$ Thus
\begin{equation*}
M(\epsilon, \alpha)=\gB/\gB(P-\epsilon, \,C-\alpha, \, H-\frac{1}{2}(\epsilon \alpha-1),\, X).
\end{equation*}
From \eqref{Ugrad}, it follows that 
\begin{equation*}
M(\epsilon, \alpha)=\bigoplus_{i\geq 0} \mK Y^i \bar{1}, \qquad \bar{1}:=1+\gB(P-\epsilon, \,C-\alpha, \, H-\frac{1}{2}(\epsilon \alpha-1),\, X).
\end{equation*}

The following lemma determines the action of $X$ on the standard basis of $M(\epsilon,\alpha)$.
\begin{lemma} \label{6July26}   
	For $n\in \N$, we have 
	\begin{equation*}
	XY^n \bar{1}=\begin{cases}
	-\frac{1}{2}n Y^{n-1}\bar{1}, & \text{if } n \text{ is even},\\
	\frac{1}{2}(\epsilon \alpha- n) Y^{n-1} \bar{1}, & \text{if } n \text{ is odd}.
	\end{cases}
	\end{equation*}
\end{lemma}
\begin{proof}
	We prove by induction on $n$. For $n=0$, we have $X\bar{1}=0$. For $n=1$, the relation $XY+YX=H$ gives $XY\bar{1}=(H -YX) \bar{1}=\frac{1}{2}(\epsilon \alpha-1) \bar{1}$. Suppose that the formula holds for some $n\geq 1$. If $n$ is odd, then using the induction hypothesis and $HY^n=Y^n(H-n)$, we obtain
	\begin{equation*}
	\begin{aligned}
	XY^{n+1}\bar{1} &=(H- YX)Y^{n}\bar{1}\\
	&= \left(\tfrac{1}{2}(\epsilon \alpha-1)-n\right)Y^n \bar{1}-\tfrac{1}{2}(\epsilon \alpha- n) Y^{n} \bar{1}\\
	&=-\tfrac{1}{2}(n+1)Y^{n} \bar{1}.
	\end{aligned}
	\end{equation*}
	If $n$ is even, then 
	\begin{equation*}
	\begin{aligned}
	XY^{n+1}\bar{1} &=(H- YX)Y^{n}\bar{1}\\
	&= \left(\tfrac{1}{2}(\epsilon \alpha-1)-n\right)Y^n \bar{1}+\tfrac{1}{2}nY^n \bar{1}\\
	&=\tfrac{1}{2}(\epsilon \alpha -n-1) Y^n \bar{1}. 
	\end{aligned}
	\end{equation*}
	This proves the lemma.
\end{proof}

\begin{remark}
	The module $M(\epsilon, \alpha)$ is a Verma module. It can be realized on $\mK[\mathsf{y}]$ by 
	\begin{equation*}
	\begin{gathered}
	Y\cdot f=\mathsf{y}f, \quad H\cdot f=\frac{1}{2}(\epsilon \alpha-1)f-\mathsf{y}f', \quad P\cdot f(\mathsf{y})=\epsilon f(-\mathsf{y}),\\
	X\cdot f=-\frac{1}{2}f'+\frac{\epsilon \alpha}{4\mathsf{y}}(f(\mathsf{y})-f(-\mathsf{y})). 
	\end{gathered}
	\end{equation*}
	The assignment $Y^n \bar{1} \mapsto \mathsf{y}^n$ identifies $M(\epsilon, \alpha)$ with $\mK[\mathsf{y}]$ endowed with the displayed action. The differential-difference formula for $X$ is well-defined, since $\mathsf{y}^{-1}(f(\mathsf{y})-f(-\mathsf{y}))\in \mK[\mathsf{y}]$ for every $f(\mathsf{y})\in \mK[\mathsf{y}]$. Moreover, the displayed formula agrees with
	the action in Lemma~\ref{6July26} on the basis
	$\{\mathsf{y}^n\mid n\in\N\}$.
	Note that $M(\epsilon, \alpha)$ is a highest weight module with highest weight $\frac{1}{2}(\epsilon \alpha-1)$. 
\end{remark}

\begin{proposition} \label{b6July26}  
	\begin{enumerate}
		\item If $\alpha \notin \{ \epsilon n \mid n\in \N^{\rm odd}\}$, then $M(\epsilon, \alpha)$ is a simple module. 
		
		\item If $\alpha=\epsilon n$ for some $n\in \N^{\rm odd}$, then $M(\epsilon,\alpha)$ has a unique maximal submodule, and the corresponding simple quotient is isomorphic to 
		\begin{equation*}
		S(\epsilon, n):=\gB/\gB(P-\epsilon, \,C-\epsilon n, \, H-\frac{1}{2}(n-1), \, X,\,Y^n). 
		\end{equation*}
	\end{enumerate}
\end{proposition}
\begin{proof}
	(1) Let $W$ be a nonzero submodule of $M(\epsilon, \alpha)$, and choose $0\neq w=\sum_{i=0}^n \beta_i Y^i \bar{1}\in W$ with $\beta_i\in \mK$ and $\beta_n \neq 0$. Since $H Y^i \bar{1}=(\hat{\alpha}-i) Y^i \bar{1}$ where $\hat{\alpha}:=\frac{1}{2}(\epsilon \alpha-1)$, the vectors $Y^i\bar1$ are $H$-eigenvectors with pairwise distinct eigenvalues. Hence every weight component of $w$ belongs to $W$.  In particular, $Y^n \bar{1}\in W$. Since $\alpha \notin \{ \epsilon n \mid n\in \N^{\rm odd} \}$, repeated application of Lemma~\ref{6July26} shows that $\bar{1}\in W$. Therefore, $W=M(\epsilon,\alpha)$, and $M(\epsilon,\alpha)$ is simple. 
	
	(2) Suppose $\alpha=\epsilon n$ for some $n\in \N^{\rm odd}$. By Lemma \ref{6July26}, $XY^n \bar{1}=0$. Let
	\begin{equation*}
	M'=\bigoplus_{i\geq n}\mK Y^i\bar1.
	\end{equation*}
	It is clear that $M'$ is stable under the action of the elements $P$, $X$, $Y$ and $H$. Thus $M'$ is a $\gB$-module. 
	
	Now let $W$ be a proper submodule of $M(\epsilon,\alpha)$. Since the
	$H$-weights are pairwise distinct, $W$ is a direct sum of weight
	spaces. Hence, if $W \neq 0$, then $Y^i\bar1\in W$ for some $i$.	 
	If $i<n$, repeated application of Lemma~\ref{6July26} shows that
	$\bar1\in W$, contradicting the assumption that $W$ is proper.
	Therefore every nonzero weight vector in $W$ has degree at least $n$,
	and thus
	$W\subseteq M'.$
	Hence $M'$ is the unique maximal submodule of $M(\epsilon,\alpha)$.
	Clearly, 
	$M(\epsilon,\alpha)/M'
	\cong S(\epsilon,n)$.
	Its simplicity follows from the maximality of $M'$.
\end{proof}

 Jacobson's lemma (\cite[Lemma 2]{Jacobson}) asserts that if $A$ and $B$ are linear
operators on a finite-dimensional vector space over a field of
characteristic zero and $[A,[A,B]]=0$, then $[A,B]$ is nilpotent. We now classify the finite-dimensional simple $\gB$-modules. 
\begin{theorem} \label{B6July26}  
	A complete set of isomorphism classes of finite-dimensional simple $\gB$-modules is
	\begin{equation*}
	\left\{[S(\epsilon,n)] \mid \epsilon \in \{\pm 1 \}, \,n\in \N^{\rm odd} \right\}.
	\end{equation*}
	In particular, since $\dim S(\epsilon, n)=n$, all finite-dimensional simple $\gB$-modules have odd dimension.
\end{theorem}
\begin{proof}
	Let $M$ be a finite-dimensional simple $\gB$-module. By Schur's lemma, the central element $C$ acts on $M$ by a scalar $\alpha \in \mK$. Using $[X,H]=-X$, we obtain $[X,[X,H]]=0$. Jacobson's lemma therefore implies that $X$ acts nilpotently on $M$. Thus there exists a nonzero vector $v\in M$ such that $Xv=0$. Since $P^2=1$, we have $(P+1)(P-1)v=0$. Assume first that $(P-1)v=0$. Since $M$ is simple, $v$ generates $M$. Given that $Pv=v$, $Xv=0$, and $Cv=\alpha v$, it follows that $M$ is a quotient of the module $M(1, \alpha)$. Assume now that $(P-1)v \neq 0$. Set $w=(P-1)v$. Then $Xw=0$, $(P+1)w=0$, and $Cw=\alpha w$. Since $M$ is simple, $w$ generates $M$. Hence $M$ is a simple quotient of the module $M(-1,\alpha)$. Thus, in either case, $M$ is a simple quotient of $M(\epsilon, \alpha)$ for some $\epsilon \in \{ \pm 1 \}$ and $\alpha \in \mK$. The result now follows from Proposition \ref{b6July26}. 
\end{proof}

\begin{theorem}
	Every finite-dimensional $\mathfrak{B}$-module is completely reducible. 
\end{theorem}
\begin{proof}
	Let $M$ be a finite-dimensional $\mathfrak{B}$-module. Since $P^2=1$, the operator $P$ is diagonalizable.
	Hence
	\begin{equation*}
	M=M_{\bar 0} \oplus M_{\bar 1}, \qquad M_{\bar 0}:=\ker(P-1), \quad M_{\bar 1}:=\ker(P+1). 
	\end{equation*}
	Because $PH=HP$ and $PX=-XP$, $PY=-YP$, the operator $H$ preserves each $M_{\bar i}$, while $X$ and $Y$ interchange $M_{\bar 0}$ and $M_{\bar 1}$. Thus $M$ acquires the structure of a
	finite-dimensional $\mathfrak{osp}(1|2)$-supermodule by declaring
	$M_{\bar 0}$ to be the even part and $M_{\bar 1}$ to be the odd part,
	with $H$ acting as an even endomorphism and $X,Y$ acting as odd
	endomorphisms.
	
	Conversely, let $V=V_{\bar 0}\oplus V_{\bar 1}$ be a finite-dimensional $\mathfrak{osp}(1|2)$-supermodule. Define an operator $P$ on $V$ by 
	\begin{equation*}
	P|_{V_{\bar 0}}=\operatorname{id}_{V_{\bar 0}}, \qquad P|_{V_{\bar 1}}=-\operatorname{id}_{V_{\bar 1}}. 
	\end{equation*}
	Then $P^2=1$, the operator $P$ commutes with the even action of $H$, and it anticommutes with the odd actions of $X$ and $Y$. Hence $V$ becomes a $\mathfrak{B}$-module. These two constructions are inverse to each other, so finite-dimensional $\mathfrak{B}$-modules are the same thing as finite-dimensional $\mathfrak{osp}(1|2)$-supermodules, and under this identification $\mathfrak{B}$-submodules correspond exactly to super $\mathfrak{osp}(1|2)$-submodules. 
	
	Djokovi\'{c} and Hochschild \cite{Djokovic-Hochschild} prove that every finite-dimensional $\mathfrak{osp}(1|2)$-supermodule is completely reducible in characteristic zero. Therefore $M$, viewed as an $\mathfrak{osp}(1|2)$-supermodule, is a direct sum of simple $\mathfrak{osp}(1|2)$-supermodules. Transporting this decomposition back through the equivalence above yields a direct sum decomposition of $M$ into simple $\mathfrak{B}$-submodules. Thus $M$ is completely reducible.  
\end{proof}

\subsection{Prime ideals of $\gB$}
Recall that a proper ideal $\gp$ of a ring $R$ is \emph{prime} if,  for all ideals $\ga$ and $\gb$ of $R$, the inclusion $\ga \gb \subseteq \gp$ implies $\ga \subseteq \gp$ or $\gb \subseteq \gp$. The set of prime ideals of $R$ is called the prime spectrum of $R$, and is denoted by $\Spec(R)$. 
A prime ideal $\gp$ is \emph{completely prime} if $R/\gp$ is a domain. The set of completely prime ideals of $R$ is denoted by $\Spec_c(R)$.  For an $R$-module $M$, its annihilator is the ideal
\begin{equation*}
\ann_R(M):=\{ r\in R \mid rm=0 \,\,\text{for all } m\in M \}. 
\end{equation*}
An ideal $\gp$ of a ring $R$ is \emph{primitive} if it is the annihilator of a simple $R$-module. The set of primitive ideals of $R$ is called the primitive spectrum of $R$, and is denoted by $\Prim(R)$.  It is well known that all primitive ideals are prime, and all maximal ideals are primitive.

For every $\alpha\in \mK$, define
\begin{equation*}
\gB^{\alpha}:=\gB/(C-\alpha)\gB.
\end{equation*}
	It follows from \eqref{UGWA} that $\gB^{\alpha}$ can be presented as a GWA in the form:
\begin{equation} \label{UaGWA} 
\gB^{\alpha} \cong \mathbf{D}[X, Y;\s,\mathtt{a}:=\frac{1}{4}(2H-\alpha P+1)], \qquad \mathbf{D}:=\mK[H, P]/\langle P^2-1 \rangle,
\end{equation}
where $\s$ is the automorphism of $\mathbf{D}$ defined by $\s(H)=H-1$
and $\s(P)=-P$. 	The algebra $\gB^{\alpha}$ is not a domain for any $\alpha \in \mK$, since $\mathbf{D}$ is not a domain.

The following result, due to Bavula \cite[Theorem 4.2]{BavulaSimpleCriteria},
provides a simplicity criterion for generalized Weyl algebras.
\begin{theorem} \label{SimpGWA}   
	Let $A=D[x, y;\s, \mathtt{a}]$ be a GWA. Then $A$ is a simple algebra if and only if
	\begin{enumerate}[label=(\roman*)]
		\item $D$ has no proper $\s$-stable ideals;
		\item no power of $\s$ is an inner automorphism of $D$;
		\item $D \mathtt{a}+D\s^n(\mathtt{a})=D$ for all $n\geq 1$;
		\item $\mathtt{a}$ is regular in $D$. 
	\end{enumerate}
\end{theorem}

The following proposition determines precisely when $\gB^\alpha$ is simple and describes its prime ideals in the exceptional cases.

\begin{proposition} \label{4July26}  
	\begin{enumerate}
		\item 	The algebra $\gB^{\alpha}$ is simple if and only if $\alpha \notin \{ \pm n \mid n\in \N^{\rm odd} \}$. 
		
		\item 	Let $\alpha \in \Z^{\rm odd}$. Then $\gB^{\alpha}$ is prime and has a unique nonzero prime ideal. 
	\end{enumerate}
\end{proposition}

\begin{proof}
	(1) We verify the hypotheses of the simplicity criterion for GWAs (Theorem \ref{SimpGWA}).

	(i) Set $e_{\pm}=\frac{1}{2}(1\pm P)$. Then $e_+$ and $e_-$ are orthogonal central idempotents of $\mathbf D$, and
	$\mathbf{D}=\mathbf{D}e_+\oplus \mathbf{D}e_- \cong \mK[H]\oplus \mK[H]$. Hence every ideal $I$ of $\mathbf D$ decomposes as
	$I=I_+e_+\oplus I_-e_-$ where $I_{\pm}$ are ideals of $\mK[H]$.
	 Suppose that $I$ is $\sigma$-stable. Since $\sigma(P)=-P$, the
	 automorphism $\sigma$ interchanges the two components. Moreover, $\sigma^2(H)=H-2$. Therefore, each $I_{\pm}$ is invariant under the nontrivial translation
	 $H\mapsto H-2$. The only ideals of $\mK[H]$ invariant under such a
	 translation are $0$ and $\mK[H]$. Thus $I_\pm\in\{0,\mK[H]\}$. The compatibility between the two components imposed by $\s$ then implies that $I\in\{0,\mathbf{D}\}$. Hence $\mathbf{D}$ has no proper $\s$-stable ideals. 
	
	(ii) Since $\mathbf{D}$ is commutative, all inner automorphisms are trivial. Hence no positive power of $\s$ is inner.

	(iii)  A direct computation yields
	\begin{equation*}
	\s^n(\mathtt{a})=\frac{1}{4}\left(2H+(-1)^{n+1}\alpha P-2n+1\right). 
	\end{equation*}
	If $n$ is even, then $\sigma^n(\mathtt{a})-\mathtt{a}=-\frac{1}{2}n$, which is a nonzero scalar. Hence $\mathbf{D}\mathtt{a}+\mathbf{D}\sigma^n(\mathtt{a})=\mathbf{D}$. 
	Assume now that $n$ is odd. Then
	\begin{equation*}
	\mathbf{D}\mathtt{a}+\mathbf{D}\sigma^n(\mathtt{a})
	=\mathbf{D}\left(2H-\alpha P+1,\; \alpha P-n\right),
	\end{equation*}
	which equals $\mathbf{D}$ if and only if $\alpha\neq \pm n$. Hence
	\begin{equation*}
	\mathbf{D}\mathtt{a}+\mathbf{D}\sigma^n(\mathtt{a})=\mathbf{D} \text{ for all } n\geq 1
	\quad\Longleftrightarrow\quad
	\alpha\notin \{\pm n \mid n\in\mathbb{N}^{\mathrm{odd}}\}.
	\end{equation*}

	(iv) Writing $\mathtt{a}=\frac{1}{4}(2H-\alpha+1)e_{+}+\frac{1}{4}(2H+\alpha+1)e_{-}$,  both components are nonzero polynomials in $\mK[H]$, hence $\mathtt{a}$ is regular in $\mathbf{D}$.
	
	Thus Theorem~\ref{SimpGWA} implies that $\gB^\alpha$ is simple exactly
	when $\alpha \notin \{ \pm n \mid n\in \N^{\rm odd} \}$. 
	
	(2) Let $\alpha=\epsilon n$ with $\epsilon \in \{\pm 1 \}$ and $n\in \N^{\rm odd}$. By part (1), $\gB^{\alpha}$ is not simple.   Let $\gB^{\alpha}_X$ and $\gB^{\alpha}_Y$ be the localizations of $\gB^{\alpha}$ at the Ore sets $\{ X^i \mid i\in \N \}$ and $\{Y^i \mid i\in \N \}$, respectively. From \eqref{UaGWA} we have
	\begin{equation*}
	\gB^{\alpha}_X \cong \mathbf{D}[X^{\pm 1};\s], \qquad \gB^{\alpha}_Y \cong \mathbf{D}[Y^{\pm 1};\s^{-1}],
	\end{equation*}
	where $\s$ is the automorphism of $\mathbf{D}$ defined by $\s(H)=H-1$
	and $\s(P)=-P$. 	
	Since $\mathbf{D}$ has no proper $\s$-stable ideals and no power of $\s$ is inner,  both localizations are simple by \cite[Theorem 1.8.5]{MR}. In particular, they are prime, and hence $\gB^\alpha$ is prime.
	
	Let $\gp$ be a nonzero prime ideal of $\gB^{\alpha}$. Since $\gB^{\alpha}_X$ and $\gB^{\alpha}_Y$ are simple,  there exist $m,n \geq 1$ such that $X^m \in \gp$ and $Y^n\in \gp$. Replacing $m$ and $n$ by their maximum, we may assume $m=n$. A direct computation yields
	\begin{equation} \label{YnXn} 
	Y^n X^n =\frac{1}{4^n} \prod_{i=1}^{n}\Big(2H+(-1)^i \alpha P+2i-1\Big)\in \gp \cap \mathbf{D}. \\
	\end{equation}
	From \eqref{UaGWA} and \eqref{DGWA}, we obtain
	\begin{equation*}
	\gB^{\alpha}=\bigoplus_{i\geq 0} \mathbf{D}Y^i \oplus \bigoplus_{i>0} \mathbf{D}X^i. 
	\end{equation*}
	Since $X^n,Y^n\in\gp$, together with \eqref{YnXn} this implies that $\gB^\alpha/\gp$ is finite-dimensional over $\mK$. By \cite[Theorem 1.38]{Lorenz}, every finite-codimensional prime ideal is maximal. Thus every nonzero prime ideal of $\gB^\alpha$ is a maximal ideal of finite codimension and consequently is the annihilator of a finite-dimensional simple $\gB^\alpha$-module. For the fixed value $\alpha=\epsilon n$, Proposition~\ref{b6July26} shows that $\gB^\alpha$ has a unique finite-dimensional simple module, namely $S(\epsilon,n)$. Hence every nonzero prime ideal of $\gB^\alpha$ is equal to $\ann_{\gB^{\alpha}}(S(\epsilon, n))$. Therefore $\gB^\alpha$ has a unique nonzero prime ideal.
\end{proof}

\begin{proposition} \label{A6July26}  
	Every nonzero ideal of $\gB$ has nonzero intersection with $Z(\gB)$. Furthermore, $\gB$ is a prime algebra. 
\end{proposition}
\begin{proof}
	Let $\CZ:=\mK[C]\setminus\{0\}$ and denote by $\gB_{\mathcal{Z}}$ the localization of $\gB$ at $\CZ$.
	 Set $\mathcal{K}:=\mK(C)$. By \eqref{UGWA}, $\gB_{\CZ}$ is a generalized Weyl algebra
	 \begin{equation*}
	 \gB_{\CZ}=\CD[X,Y;\s, \mathtt{a}:=\frac{1}{4}(2H-CP+1)], \quad \CD:=\mathcal{K}[H,P]/\langle P^2-1\rangle
	 \end{equation*}
	 where $\s$ is the automorphism of $\CD$ determined by $\s(C)=C$, $\s(H)=H-1$ and $\s(P)=-P$. We prove that $\gB_{\CZ}$ is simple by applying Theorem \ref{SimpGWA}. 
	 As in the proof of  Proposition \ref{4July26}, replacing $\mathbf{D}$ by $\CD$ and $\mK$ by $\mathcal{K}$, one verifies that $\CD$ has no proper $\s$-stable ideals, no powers of $\s$ is inner, and $\mathtt{a}$ is regular. It remains to show that $\CD\mathtt{a}+\CD\s^n(\mathtt{a})=\CD$ for all $n \geq 1$. 
	 
	 A direct computation gives
	 \begin{equation*}
	 \s^n(\mathtt{a})=\frac{1}{4}\left(2H+(-1)^{n+1}C P-2n+1\right).
	 \end{equation*}
	 If $n$ is even, then $\sigma^n(\mathtt{a})-\mathtt{a}=-\frac{1}{2}n$, which is a nonzero scalar. Hence  $\CD\mathtt{a}+\CD\sigma^n(\mathtt{a})=\CD$. Suppose now that $n$ is odd. Then $\s^{n}(\mathtt{a})-\mathtt{a}=\frac{1}{2}(CP-n)$, so that
	 \begin{equation*}
	 \CD\mathtt{a}+\CD\sigma^n(\mathtt{a})
	 =\CD\left(2H-CP+1,\; CP-n\right). 
	 \end{equation*}
	 Since $P^2=1$, we have $(CP+n)(C P-n)=C^2-n^2 \in \CD\mathtt{a}+\CD\sigma^n(\mathtt{a})$. As $C^2-n^2$ is a nonzero element of the field $\mathcal{K}$, it is invertible. Therefore, $\CD\mathtt{a}+\CD\sigma^n(\mathtt{a})=\CD$. Hence $\gB_{\CZ}$ is simple. It follows that every nonzero ideal of $\gB$ must meet $\CZ$. Since $\CZ=\mK[C]\setminus\{0\}\subseteq Z(\gB)$, every nonzero ideal of $\gB$ intersects $Z(\gB)=\mK[C]$ nontrivially.	 
	 
	 Finally, let $\mathfrak a,\mathfrak b$ be nonzero ideals of $\gB$ satisfying $\ga \gb=0$, then their extensions in $\gB_{\CZ}$ are nonzero ideals whose product is zero, contradicting simplicity of $\gB_{\CZ}$. Hence $\gB$ is prime.
\end{proof}

\begin{theorem} \label{SpecB}  
	Set 
	$\ga_{\epsilon,n}:=\ann_\gB(S(\epsilon,n))$ for $\epsilon\in\{\pm1\}$ and $n\in\mathbb{N}^{\mathrm{odd}}$.

	\begin{enumerate}
		\item The prime spectrum of $\gB$ is 
		\begin{equation*}
			\Spec(\gB)=\{ \langle 0 \rangle \} \,\cup \,\{\langle C-\alpha \rangle \mid \alpha \in \mK \}\,\cup \, \{ \ga_{\epsilon,n}\mid \epsilon\in\{\pm 1\},  n\in \N^{\rm odd} \}. 
		\end{equation*}
		The inclusions among these prime ideals are indicated in the following diagram
		
		\begin{equation*}
		\begin{tikzpicture}[scale=0.6, every node/.style={font=\small}]
		
		\filldraw[black] (0,0) circle (2pt) 
		node[below=2pt] {$\langle 0 \rangle$};
		
		\draw[thick, black] (-5.5,2) -- (5.5,2);
		
		\foreach \x in {-5.5,-5.4,-5.3,...,5.5} {
			\draw[gray!60] (0,0) -- (\x,2);
		}
		
		\foreach \x in {-5.5,-5.4,-5.3,...,5.5} {
			\filldraw[black] (\x,2) circle (0.8pt);
		}
		
		\foreach \x/\lab in {
			-5/{\ga_{-,5}},
			-3/{\ga_{-,3}},
			-1/{\ga_{-,1}},
			1/{\ga_{+,1}},
			3/{\ga_{+,3}},
			5/{\ga_{+,5}}
		}
		{
			\filldraw[black] (\x,2) circle (1pt);
			\draw[gray!70] (\x,2) -- (\x,3);
			\filldraw[black] (\x,3) circle (1.2pt)
			node[above=3pt] {$\lab$};
		}
		
		\node at (-6,2) {$\cdots$};
		\node at (6,2) {$\cdots$};
		\node at (-6,3) {$\cdots$};
		\node at (6,3) {$\cdots$};
		
		\node[right] at (6.4,2) 
		{$\{\langle C-\alpha\rangle\}$};
		
		\end{tikzpicture}
		\end{equation*}

		\item The maximal spectrum of $\gB$ is
		\begin{equation*}
		\Max(\gB)= \{\langle C-\alpha \rangle \mid \alpha \in \mK \setminus \Z^{\rm odd}\}\,\, \cup \,\, \{ \ga_{\epsilon ,n} \mid \epsilon\in \{\pm 1\},  n\in \N^{\rm odd} \}. 
		\end{equation*}
		
		\item Every nonzero prime ideal of $\gB$ is primitive; that is, 
		\begin{equation*}
		\Prim(\gB)=\Spec(\gB) \setminus \{\langle 0 \rangle \}.
		\end{equation*}
		
		\item The set of completely prime ideals of $\gB$ is $\Spec_c(\gB)=\{ \ga_{-,1}, \,\ga_{+,1} \}$.  
	\end{enumerate}
\end{theorem}
\begin{proof}
	(1) By Proposition \ref{A6July26}, $\gB$ is prime, so $\langle 0 \rangle$ is a prime ideal.	 Let $\gp$ be a nonzero prime ideal of $\gB$. Again by Proposition \ref{A6July26}, $\gp \cap Z(\gB) \neq 0$. Since $Z(\gB)=\mK[C]$, there exists $\alpha\in\mK$ such that $\langle C-\alpha\rangle \subseteq \gp$.

	If $\alpha \notin \Z^{\rm odd}$, then Proposition~\ref{4July26}(1) implies that
	$\gB^\alpha=\gB/\langle C-\alpha\rangle$ is simple. Hence
	$\langle C-\alpha\rangle$ is maximal, and therefore $\gp=\langle C-\alpha\rangle$. 
	
	Now suppose that $\alpha\in\Z^{\rm odd}$. By Proposition~\ref{4July26}(2),
	$\langle C-\alpha\rangle$ is prime and there exists a unique prime ideal properly
	containing it. Writing $\alpha=\epsilon n$ with
	$\epsilon\in\{\pm1\}$ and $n\in\N^{\rm odd}$, this unique prime ideal is $\ann_\gB(S(\epsilon,n))=\ga_{\epsilon,n}.$ Therefore all nonzero prime ideals are exactly those listed in the statement.

    (2) This follows immediately from the inclusions of prime ideals described in part (1).
    
    (3) Since $Z(\mathfrak B)=\mK[C]$ is a polynomial algebra by
    Proposition~\ref{5July26}, the zero ideal is not primitive. Every
    maximal ideal is primitive, so it remains to consider the prime ideals
    $\langle C-\alpha\rangle$ with $\alpha\in\mathbb Z^{\mathrm{odd}}$.    
    Fix such an $\alpha$, and write $\alpha=\epsilon n$ with
    $\epsilon\in\{\pm1\}$ and $n\in\mathbb N^{\mathrm{odd}}$. Since
    $\mathfrak B$ is a Jacobson ring, the prime ideal
    $\langle C-\alpha\rangle$ is the intersection of the primitive ideals
    containing it. By part~(1), the only prime ideals containing
    $\langle C-\alpha\rangle$ are
    $\langle C-\alpha\rangle$ and $\mathfrak a_{\epsilon,n}$.
    Since $\mathfrak a_{\epsilon,n}$ is maximal, it is primitive. Hence
    $\langle C-\alpha\rangle$ itself is primitive.    
    Thus every nonzero prime ideal of $\mathfrak B$ is primitive.

	(4)  Since $\gB$ is not a domain, $\langle0\rangle$ is not completely prime.  For every $\alpha\in\mK$, the algebra $\gB^\alpha$ is not a domain by \eqref{UaGWA}, so $\langle C-\alpha\rangle$ is not completely prime. Finally, $\gB/\ga_{\epsilon,n} \cong M_n(\mK)$ is a matrix algebra, which is a domain if and only if $n=1$. Hence $\ga_{-,1}$ and $\ga_{+,1}$ are completely prime, whereas $\ga_{\epsilon,n}$ is not completely prime for $n>1$.
\end{proof}

\section{Finite-dimensional modules and prime ideals of $\mathfrak{U}$}
\label{Uspec}  

In this section we transfer the structural results obtained for $\mathfrak B$ to $\mathfrak U$. We first show that the finite-dimensional module theory of $\mathfrak U$ is completely determined by that of $\mathfrak B$, and then determine the prime spectrum of $\mathfrak U$ using the localization description established in Proposition \ref{b23Aug26}.

\subsection{Finite-dimensional $\mathfrak{U}$-modules}
The following theorem provides explicit relation between $\mathfrak B$ and $\mathfrak U$ on finite-dimensional modules.

\begin{theorem} \label{A24Aug26}   
Restriction along the inclusion $\mathfrak B\subset\mathfrak U$ induces a bijection between the finite-dimensional $\mathfrak U$-modules and the finite-dimensional $\mathfrak B$-modules. More precisely, every finite-dimensional $\mathfrak B$-module admits a unique $\mathfrak U$-module structure extending its given $\mathfrak B$-action. Under this correspondence, the $\mathfrak B$-submodules and $\mathfrak U$-submodules coincide. In particular, simplicity is preserved.
\end{theorem}

\begin{proof}
	Let $M$ be a finite-dimensional $\mathfrak{B}$-module. Since $[X,H]=-X$, we have $[X,[X,H]]=0$. Jacobson's lemma implies that $X$ acts nilpotently on $M$. 
	
	Since $X$ is nilpotent, $1+\sqrt{\lambda}^{-1}X$ is invertible on $M$, with inverse given by the finite sum
	\begin{equation*}
	(1+\sqrt{\lambda}^{-1}X)^{-1}=\sum_{n=0}^{N-1} (-\sqrt{\lambda}^{-1} X)^n
	\end{equation*}
	for $N$ sufficiently large. Define endomorphisms of $M$ by 
	\begin{equation*}
	g:=(1+\sqrt{\lambda}^{-1}X)^{-1}P, \quad x:=Xg, \quad y:=(2\lambda Y-\sqrt{\lambda}H)P. 
	\end{equation*}
	These  are exactly the formulas expressing the generators $g$, $x$, $y$ of $\mathfrak{U}$ in terms of $P$, $H$, $X$, $Y$, together with the inverse of $1+\sqrt{\lambda}^{-1}X$. These endomorphisms satisfy the defining relations \eqref{DefU}. Hence $M$ acquires a $\mathfrak U$-module structure extending its given $\mathfrak B$-action.
	
	This extension is unique. Indeed, on any finite-dimensional $\mathfrak{U}$-module the restricted $\mathfrak B$-action again has nilpotent $X$, and the relations
	\begin{equation*}
	P=(1+\sqrt{\lambda}^{-1}X)g, \qquad X=xg^{-1}
	\end{equation*}
	imply $g=(1+\sqrt{\lambda}^{-1}X)^{-1}P$ and $x=Xg$. By \eqref{YY}, we also have $y=(2\lambda Y-\sqrt{\lambda}H)P$. Thus the actions of $g$, $x$, and $y$ are uniquely determined by the restricted $\mathfrak B$-action. Hence the extension constructed above is unique.
	
	Thus restriction and extension are inverse on finite-dimensional modules. Moreover, if $W \subseteq M$ is a $\mathfrak{U}$-submodule, then it is automatically a $\mathfrak{B}$-submodule. Conversely, if $W\subseteq M$ is a $\mathfrak{B}$-submodule, then the formulas above show that $W$ is stable under
	$g$, $x$, and $y$, hence under $\mathfrak{U}$. Therefore the lattices of submodules coincide, and simplicity is preserved. 
	\end{proof}

\begin{corollary} \label{B24Aug26}  
	The category
	of finite-dimensional $\mathfrak{U}$-modules is equivalent to the
	category of finite-dimensional $\mathfrak{B}$-modules. Under this
	equivalence, simple objects correspond to simple objects.
\end{corollary}
\begin{proof}
	By Theorem~\ref{A24Aug26}, restriction from $\mathfrak{U}$ to
	$\mathfrak{B}$ is bijective on objects, and that every finite-dimensional
	$\mathfrak{B}$-module admits a unique extension to a finite-dimensional
	$\mathfrak{U}$-module.	Thus restriction is essentially surjective.
	Moreover, let $M$ and $N$ be finite-dimensional $\mathfrak U$-modules. A linear map $f\colon M\to N$ is $\mathfrak U$-linear if and only if it is $\mathfrak B$-linear. Indeed, the generators $g$, $x$, and $y$ of $\mathfrak U$ are determined by the $\mathfrak B$-action through the formulas in the proof of Theorem~\ref{A24Aug26}. Hence restriction is fully faithful and therefore an equivalence of categories. The final assertion follows from the fact that $\mathfrak B$-submodules and $\mathfrak U$-submodules coincide.
\end{proof}

\subsection{Prime ideals of $\mathfrak U$}
The following proposition reduces the classification of prime ideals of $\mathfrak U$ to that of $\mathfrak B$.
\begin{proposition}  \label{a5July26}  
	There is an inclusion-preserving bijection
	\begin{equation*}
	\Spec(\gB) \longrightarrow \Spec(\mathfrak{U}), \qquad \gp \mapsto \gp \mathfrak{U},
	\end{equation*}
	whose inverse is given by $\gq \mapsto \gq \cap \gB$. 
\end{proposition}

\begin{proof}
	Recall that $w=\lambda-X^2$, and set $S=\{ w^n \mid n \in \N \}$. Let $I$ be a two-sided ideal of $\mathfrak{B}$. We first show that every proper two-sided ideal of $\mathfrak B$ is disjoint from $S$. Let $I$ be a proper ideal of $\mathfrak B$, and suppose, to the contrary, that $w^n\in I$ for some $n\in\mathbb N$. Choose $n$ minimal with this property. Since $[H,X]=X$, we have 
	\begin{equation*} 
	[H,w]=[H,\lambda-X^2]=2w-2\lambda. 
	\end{equation*} 
	Hence 
	\begin{equation*} 
	[H,w^n]=2n w^n-2n\lambda w^{n-1}.
	 \end{equation*} 
	 Since $w^n\in I$ and $I$ is a two-sided ideal,
	  \begin{equation*} 
	  [H,w^n]-2n w^n=-2n\lambda w^{n-1}\in I.
	  \end{equation*}
	   As $\lambda\neq0$, it follows that $w^{n-1}\in I$. By the minimality of $n$, we must have $n=1$. Thus $w\in I$, and consequently  $[H,w]=2w-2\lambda\in I$. Therefore $2\lambda\in I$, and hence $1\in I$, a contradiction. Thus $I\cap S=\varnothing.$
	   
	   By \cite[Proposition 2.1.16(vii)]{MR}, extension and contraction give an inclusion-preserving bijection between $\Spec(\mathfrak B_w)$ and the prime ideals of $\mathfrak B$ disjoint from $S$. By the preceding argument, every prime ideal of $\mathfrak B$ is disjoint from $S$. Thus extension and contraction induce a bijection 
	   \begin{equation*} 
	   \Spec(\mathfrak B)\xrightarrow{\sim}\Spec(\mathfrak B_w). 
	   \end{equation*} 
	   Finally, Proposition \ref{b23Aug26} identifies $\mathfrak B_w$ with $\mathfrak U$. Under this identification, the extension and contraction maps are precisely \begin{equation*} 
	   \mathfrak p\mapsto\mathfrak p\mathfrak U, \qquad \mathfrak q\mapsto\mathfrak q\cap\mathfrak B. 
	   \end{equation*} 
	   This proves the proposition.
\end{proof}

\begin{theorem} \label{SpecU}   
	Set $\gm_{\epsilon, n}:=\ann_{\mathfrak{U}}(S(\epsilon,n))$ for $\epsilon\in \{\pm 1\}$ and $n\in \N^{\rm odd}$. 
	\begin{enumerate}
	\item The prime spectrum of $\mathfrak{U}$ is 
	\begin{equation*}
	\Spec(\mathfrak{U})=\{ \langle 0 \rangle \} \,\cup \,\{\langle C-\alpha \rangle \mid \alpha \in \mK \}\,\cup \, \{ \gm_{\epsilon,n}\mid \epsilon\in\{\pm 1\},  n\in \N^{\rm odd} \}. 
	\end{equation*}
	The inclusions of prime ideals are indicated in the following diagram
	
	\begin{equation*}
	\begin{tikzpicture}[scale=0.6, every node/.style={font=\small}]
	
	\filldraw[black] (0,0) circle (2pt) 
	node[below=2pt] {$\langle 0 \rangle$};
	
	\draw[thick, black] (-5.5,2) -- (5.5,2);
	
	\foreach \x in {-5.5,-5.4,-5.3,...,5.5} {
		\draw[gray!60] (0,0) -- (\x,2);
	}
	
	\foreach \x in {-5.5,-5.4,-5.3,...,5.5} {
		\filldraw[black] (\x,2) circle (0.8pt);
	}
	
	\foreach \x/\lab in {
		-5/{\gm_{-,5}},
		-3/{\gm_{-,3}},
		-1/{\gm_{-,1}},
		1/{\gm_{+,1}},
		3/{\gm_{+,3}},
		5/{\gm_{+,5}}
	}
	{
		\filldraw[black] (\x,2) circle (1pt);
		\draw[gray!70] (\x,2) -- (\x,3);
		\filldraw[black] (\x,3) circle (1.2pt)
		node[above=3pt] {$\lab$};
	}
	
	\node at (-6,2) {$\cdots$};
	\node at (6,2) {$\cdots$};
	\node at (-6,3) {$\cdots$};
	\node at (6,3) {$\cdots$};
	
	\node[right] at (6.4,2) 
	{$\{\langle C-\alpha\rangle\}$};
	
	\end{tikzpicture}
	\end{equation*}
	
	\item The maximal spectrum of $\mathfrak{U}$ is
	\begin{equation*}
	\Max(\mathfrak{U})= \{\langle C-\alpha \rangle \mid \alpha \in \mK \setminus \Z^{\rm odd}\}\,\, \cup \,\, \{ \gm_{\epsilon ,n} \mid \epsilon\in \{\pm 1\},  n\in \N^{\rm odd} \}. 
	\end{equation*}
	
	\item Every nonzero prime ideal of $\mathfrak{U}$ is primitive; that is, 
	\begin{equation*}
	\Prim(\mathfrak{U})=\Spec(\mathfrak{U}) \setminus \{\langle 0 \rangle \}.
	\end{equation*}
	
	\item The set of completely prime ideals of $\mathfrak{U}$ is $\Spec_c(\mathfrak{U})=\{ \gm_{-,1}, \, \gm_{+,1} \}$.  
\end{enumerate}
\end{theorem}

\begin{proof}
	(1) By Proposition \ref{a5July26}, extension and contraction give an inclusion-preserving bijection
	$\Spec(\mathfrak B)\xrightarrow{\sim}\Spec(\mathfrak U)$.
	The description of $\Spec(\mathfrak B)$ given in Theorem \ref{SpecB} therefore yields the asserted description of $\Spec(\mathfrak U)$.
	
	(2) This follows from the inclusions of primes in statement (1). 
	
	(3) By \cite[Corollary 9.1.8]{MR}, $\mathfrak U$ is a Jacobson
	algebra satisfying the Nullstellensatz. Since $\mathfrak U$ has a
	nontrivial centre, the zero ideal is not primitive. Every maximal ideal is primitive, so it remains to consider the prime ideals
	$\langle C-\alpha\rangle$ with $\alpha\in\Z^{\rm odd}$. Write $\alpha=\epsilon n$ with
	$\epsilon\in\{\pm1\}$ and $n\in\mathbb N^{\rm odd}$. Since $\mathfrak U$
	is a Jacobson ring, every prime ideal is the
	intersection of the primitive ideals containing it. By part~(1), the only
	prime ideals containing $\langle C-\alpha\rangle$ are
	$\langle C-\alpha\rangle$ and $\mathfrak m_{\epsilon,n}$. Since
	$\mathfrak m_{\epsilon,n}$ is maximal, it is primitive. Hence
	$\langle C-\alpha\rangle$ is primitive as well. Therefore every nonzero
	prime ideal of $\mathfrak U$ is primitive.
	
	(4) Since $\mathfrak{U}$ is not a domain, the zero ideal is not completely
	prime.
	  Let $\alpha\in\mK$. Consider the factor algebra $\mathfrak{U}^{\alpha}:=\mathfrak{U}/\langle C- \alpha \rangle$. In $\mathfrak{U}^{\alpha}$ we have $(P+1)(P-1)=0$. Moreover, neither $P+1$ nor $P-1$ vanishes in $\mathfrak U^\alpha$.
	  Hence $\mathfrak U^\alpha$ has nonzero zero divisors and is therefore not a domain. Thus no ideal
	  $\langle C-\alpha\rangle$ is completely prime. Finally, since $S(\epsilon,n)$ is a finite-dimensional simple $\mathfrak U$-module over the algebraically closed field $\mK$, the density theorem gives
	  \begin{equation*}
	  \mathfrak U/\mathfrak m_{\epsilon,n}
	  \cong
	  \operatorname{End}_{\mK}\bigl(S(\epsilon,n)\bigr).
	  \end{equation*}
	  This matrix algebra is a domain if and only if
	  $\dim S(\epsilon,n)=1$.
	  As $\dim S(\epsilon,n)=n$, it follows that
	  $\mathfrak m_{\epsilon,1}$ is completely prime, whereas
	  $\mathfrak m_{\epsilon,n}$ is not completely prime for $n>1$.
	  Therefore
	  $\Spec_c(\mathfrak U)=\{\mathfrak m_{-,1},\mathfrak m_{+,1}\}.$
\end{proof}

\section{The classical quotient ring of $\mathfrak{U}$} \label{Quotient}  

In this section, we determine the classical quotient rings of $\mathfrak{U}$ and some of its subalgebras. The argument begins with localization at powers of $x$ and then reduces the computation to a suitable subalgebra whose quotient ring can be described explicitly.
\subsection{The localization $T_x$}
Recall that $T$ denotes the subalgebra of $\mathfrak{U}$ generated by $g^{\pm 1}$ and $x$, see \eqref{Tgx}. Since $gx=-xg$, the element $x$ is normal in $T$.
Let $T_x$ denote the localization of $T$ at the powers of $x$.

\begin{lemma}\label{lem:Tx-center}
	Let $u:=g^2$ and $v:=x^2$. Then
	\begin{equation}\label{ZDx}
	Z(T_x)\cong \Bbbk[u^{\pm1},v^{\pm1}] \big/\langle v-\lambda(1-u)\rangle
	\cong \Bbbk[u^{\pm1}]_{(1-u)}.
	\end{equation}
	where $\mK[u^{\pm 1}]_{(1-u)}$ denotes the localization of $\mK[u^{\pm 1}]$ at the powers of $1-u$. 
	Moreover,
	\begin{equation}\label{Dxz}
	T_x=\bigoplus_{i,j=0}^1 Z(T_x)g^ix^j.
	\end{equation}
	In particular, $T_x$ is a free $Z(T_x)$-module of rank $4$.
\end{lemma}
\begin{proof}
	The elements $u$ and $v$ are central in $T_x$. Moreover, the defining relation of $\mathfrak U$ gives
	\begin{equation*}
	v=x^2=\lambda(1-g^2)=\lambda(1-u).
	\end{equation*} 
	Every element of $T_x$ can be written uniquely in the form \begin{equation*} a=z_0+z_1g+z_2x+z_3gx, \qquad z_0,z_1,z_2,z_3\in \Bbbk[u^{\pm1},v^{\pm1}] \big/ \langle v-\lambda(1-u)\rangle. \end{equation*} Indeed, the relations $g^2=u$, $x^2=v$, and $gx=-xg$ reduce every Laurent monomial in $g$ and $x$ to one of the four forms $1$, $g$, $x$, and $gx$.	
	Now suppose that $a$ is central. Comparing $ga$ and $ag$, and using
	$gx=-xg$, gives $z_2=z_3=0.$
	Thus $a=z_0+z_1g$. Comparing $xa$ and $ax$ then gives $z_1=0.$
	Therefore
	\begin{equation*}
	Z(T_x)
	=
	\Bbbk[u^{\pm1},v^{\pm1}]
	\big/
	\langle v-\lambda(1-u)\rangle.
	\end{equation*}
	
	Since $v=\lambda(1-u)$ and $v$ is invertible in $T_x$,  we obtain the second
	isomorphism in \eqref{ZDx}. The decomposition \eqref{Dxz} is the preceding
	normal form rewritten over the centre.
\end{proof}

A set of $n\times n$ \emph{matrix units} in a ring $S$ is a collection
$\{e_{ij}\mid 1\leq i,j\leq n\}\subseteq S$ satisfying
\begin{equation*}
\sum_{i=1}^n e_{ii}=1,
\qquad
e_{ij}e_{k\ell}=\delta_{jk}e_{i\ell}
\end{equation*}
for all $i,j,k,\ell$. The existence of such a set of matrix units
realizes $S$ as a full matrix ring over a suitable subring. More precisely,
set
\begin{equation*}
R:=
\left\{
\sum_{k=1}^n e_{k1}ae_{1k}
\;\middle|\;
a\in S
\right\}.
\end{equation*}
Then $R$ is a subring of $S$, and
$S\cong M_n(R)$. Explicitly, the map
\begin{equation} \label{Matrixiso}
\varphi:S\longrightarrow M_n(R),
\qquad
a\longmapsto (r_{ij}),
\quad
r_{ij}:=\sum_{k=1}^n e_{ki}ae_{jk},
\end{equation}
is a ring isomorphism; see \cite[Proposition 1.1.3]{Rowen}.

The subring $R$ can also be identified with the centralizer of the matrix
units. Define
\begin{equation*}
C_S(\{e_{ij}\})
:=
\{c\in S\mid ce_{ij}=e_{ij}c
\text{ for all }i,j\}.
\end{equation*}
We first show that $R\subseteq C_S(\{e_{ij}\})$. If
$r=\sum_{k=1}^n e_{k1}ae_{1k}\in R$, then, for all $i,j$,
\begin{equation*}
e_{ij}r
=
\sum_{k=1}^n e_{ij}e_{k1}ae_{1k}
=
\sum_{k=1}^n \delta_{jk}e_{i1}ae_{1k}
=
e_{i1}ae_{1j},
\end{equation*}
while
\begin{equation*}
re_{ij}
=
\sum_{k=1}^n e_{k1}ae_{1k}e_{ij}
=
\sum_{k=1}^n \delta_{ki}e_{k1}ae_{1j}
=
e_{i1}ae_{1j}.
\end{equation*}
Hence $e_{ij}r=re_{ij}$ for all $i,j$, and therefore
$r\in C_S(\{e_{ij}\})$.

Conversely, let $c\in C_S(\{e_{ij}\})$. Since
$\sum_{k=1}^n e_{k1}e_{1k}=1$, we have
\begin{equation*}
c
=
\left(\sum_{k=1}^n e_{k1}e_{1k}\right)c
=
\sum_{k=1}^n e_{k1}(e_{1k}c)
=
\sum_{k=1}^n e_{k1}(ce_{1k})
=
\sum_{k=1}^n e_{k1}ce_{1k}.
\end{equation*}
Thus $c\in R$, and consequently
$R=C_S(\{e_{ij}\}).$

The following lemma shows that $T_x$ is a matrix algebra over its centre.
\begin{lemma} \label{a24Jun26}  
	Let $\CA:=Z(T_x)\cong \mK[u^{\pm 1}]_{(1-u)}$, where $u=g^2$. Then
	 $T_x \cong M_2(\CA)$. 
\end{lemma}

\begin{proof}
	Set $e_1:=\frac{1}{2}(1-P)$ and $e_2:=\frac{1}{2}(1+P)$. 
	Since $P^2=1$, the elements $e_1$ and $e_2$ are orthogonal
	idempotents satisfying $e_1+e_2=1$. Recall that $X=xg^{-1}$. Since both $x$ and $g$ are invertible in $T_x$, the element
	$X$ is invertible as well. Moreover, the relation $PX=-XP$ gives
	\begin{equation*}
	e_1X=Xe_2, \qquad e_2X=Xe_1.
	\end{equation*}
	Define
	\begin{equation} \label{matu} 
	e_{11}:=e_1, \quad e_{12}:= e_1 X, \quad e_{21}:=e_2X^{-1}, \quad e_{22}:=e_2
	\end{equation}
	One checks directly that $\{e_{ij}\mid 1\leq i,j\leq 2 \}$ is a set of matrix units in $T_x$. Hence $T_x \cong M_2(R)$ for a suitable subring $R$. 
	
	By the construction of the matrix units, we may take 
	\begin{equation*}
	R=\{ e_{11}ae_{11}+e_{21}ae_{12} \mid a\in T_x \}.
	\end{equation*}
	 Since every element of $\CA=Z(T_x)$ commutes with the matrix units, we have $\CA \subseteq R$. Conversely, by \eqref{Dxz}, every element of $T_x$ has a unique
	expression $a=z_0+z_1 g+z_2 x+z_3 gx$ where $z_i \in \CA$. A direct calculation yields
	\begin{equation*}
	e_{11}ae_{11}+e_{21}ae_{12}=z_0- z_1 g^2-\sqrt{\l} \,z_2 x^2. 
	\end{equation*}
	Since $g^2, x^2\in \CA$, the right-hand side belongs to $\CA$. Thus $R\subseteq \CA$, so
	 $R=\CA$. Consequently, $T_x\cong M_2(\CA)$.
\end{proof}

Let $\CC_R$ denote the set of regular elements of a ring $R$. Then $\CC_R$ is a multiplicative set. If $\CC_R$ is a left (resp. right) Ore set, the localization $Q_{\ell}(R):=\CC_R^{-1}R$ (resp. $Q_{r}(R):=R \CC_R^{-1}$) is called the left (resp. right) quotient ring of $R$. If $\CC_R$ is an Ore set (that is, both a left and a right Ore set), then $Q_{\ell}(R)$ and $Q_r(R)$ are isomorphic, and the ring $Q(R):=\CC_R^{-1}R=R \CC_R^{-1}$ is called the \emph{classical quotient ring} of $R$. 

\begin{corollary}
The classical quotient ring of $T$ is 
\begin{equation*}
Q(T) \cong M_2(\mK(u)).
\end{equation*}
\end{corollary}
\begin{proof}
	Since $T$ is Noetherian, the set of regular elements satisfies both the left and right Ore conditions. Therefore the classical quotient ring $Q(T)$ exists. The element $x$ is regular in $T$,  so $Q(T)=Q(T_x)$. By Lemma~\ref{a24Jun26}, $T_x \cong M_2(\CA)$. Therefore \cite[Corollary 3.1.5]{MR} gives $Q(T_x) \cong M_2(Q(\CA))$. Since $\CA$ is a commutative domain, $Q(\CA)$ is its field of fractions, namely $\mK(u)$. Hence $Q(T)=Q(T_x) \cong M_2(\mK(u))$. 
\end{proof}

\subsection{The Hopf subalgebra $\gA$}
Recall that $\gA=T[s;\d_1]$ is the subalgebra of $\mathfrak{U}$ generated by $g^{\pm 1}$, $x$, and $s$. Since $xg=-gx$ and $xs=(s+2\l g^2)x$, the element $x$ is normal in $\gA$.  Let $\gA_x$ denote the localization of $\gA$ at the powers of $x$. Then 
\begin{equation*}
\gA_x=T_x[s;\d_1], 
\end{equation*}
where $\d_1$ extends to a derivation of $T_x$ determined by
$\d_1(g)=2x^2 g$ and $\d_1(x)=-2\l g^2 x$. 

The matrix units 
\eqref{matu} lie in $T_x \subseteq \gA_x$. Set
\begin{equation*}
h:=H+e_2. 
\end{equation*}
Using the expression for $H$ in \eqref{HH}, we obtain
\begin{equation} \label{Defh}  
h=-\frac{1}{2\l}\left(s+\sqrt{\l} xg \right)+\frac{1}{2}\left(1+g+\frac{1}{\sqrt{\l}}x \right). 
\end{equation}
Since $g^2$ commutes with $x$, we have $[h, g^2]=-\frac{1}{2\l}[s, g^2]$. From \eqref{xsg}  it follows that $[s,g^2]=4x^2g^2=4\l(1-g^2)g^2$, and therefore
\begin{equation} 
[h, g^2]=-2(1-g^2)g^2.  \label{hg2} 
\end{equation}

\begin{lemma} \label{hmat} 
	The element $h$ commutes with the matrix units $e_{ij}$, $1\leq i,j\leq 2$. 
\end{lemma}
\begin{proof}
	Recall that $e_{11}=e_1$, $e_{12}=e_1X$, $e_{21}=e_2X^{-1}$ and $e_{22}=e_2$. 
	Since $e_1e_2=e_2e_1=0$, we have $[h, e_1]=[H+e_2, e_1]=[H,e_1]$. Since $H$ commutes with $P$, it commutes with $e_1=\frac{1}{2}(1-P)$ and $e_2=\frac{1}{2}(1+P)$. Hence
	\begin{equation*}
	[h, e_1]=0, \qquad [h, e_2]=0.
	\end{equation*}	
	Using $[H,X]=X$ and $e_2X=Xe_1$, we compute
	\begin{equation*}
	[h, X]=[H+e_2, X]=X+e_2X-Xe_2=2Xe_1. 
	\end{equation*}
	Therefore,
	\begin{equation*}
	[h, e_{12}]=[h, e_1X]=e_1[h, X]=2e_1Xe_1=2Xe_2e_1=0. 
	\end{equation*}
	Similarly, since $[h,X^{-1}]
	=-X^{-1}[h,X]X^{-1}
	=-2e_1X^{-1}$, we get
	\begin{equation*}
	[h, e_{21}]=[h, e_2X^{-1}]=e_2[h, X^{-1}]=-2e_2e_1X^{-1}=0. 
	\end{equation*}
	Thus $h$ commutes with all the matrix units.
\end{proof}

Recall from Lemma \ref{lem:Tx-center} that $\CA=Z(T_x) \cong \mK[u^{\pm 1}]_{(1-u)}$ where $u=g^2$. The next proposition identifies $\gA_x$ as a matrix algebra over a skew polynomial ring. 
\begin{proposition} \label{b24Jun26}  
	\begin{enumerate}
		\item 	Let $\CB$ be the subalgebra of $\gA_x$ generated by $\CA$ and $h$. Then
		\begin{equation*}
		\gA_x \cong M_2(\CB), \quad \CB \cong \CA[h;\d] \cong \mK[u^{\pm 1}]_{(1-u)}[h;\d],
		\end{equation*}
		where $\d$ is the derivation of $\CA$ determined by $\d(u)=-2(1-u)u$. 
		
		\item The centres of $\gA_x$ and $\gA$ are trivial: $Z(\gA_x)=Z(\gA)=\mK$. 
		
		\item The classical quotient ring of $\gA$ is 
		\begin{equation*}
		Q(\gA) \cong M_2(\mathbf{F})
		\end{equation*}
		where $\mathbf{F}:=\Frac(A_1)$ is the Weyl skew field. 
	\end{enumerate}

\end{proposition}
\begin{proof}
	(1) Recall that $\gA_x=T_x[s;\d_1].$ From \eqref{Defh}, we have
	\[
	s=-2\lambda h+
	\lambda\left(1+g+\frac1{\sqrt{\lambda}}x\right)
	-\sqrt{\lambda}xg .
	\]
	The second and third terms on the right-hand side belong to $T_x$.
	Hence $s$ and $h$ generate the same algebra over $T_x$, and therefore
	$\gA_x=T_x[h;\d]$, where $\d=\operatorname{ad}_h|_{T_x}$.
	
	The matrix units in \eqref{matu} lie in $\gA_x$, and by Lemma \ref{hmat}, $h$ commutes with each of them. We first determine the centralizer of these matrix units. 
	Let 
	\begin{equation*}
	a=\sum_{k=0}^n d_k h^k \in \gA_x, \quad d_k\in T_x. 
	\end{equation*}
	Suppose that $a$ commutes with all $e_{ij}$. Then
	\[0=[a,e_{ij}]=\sum_{k=0}^n[d_k,e_{ij}]h^k.\]
	By the uniqueness of the Ore extension expansion in $T_x[h;\d]$, we obtain
	$[d_k,e_{ij}]=0$
	for every $k$ and every $i,j$. Hence
	\[
	C_{\gA_x}(\{e_{ij}\})
	=
	\bigoplus_{k\geq0}
	C_{T_x}(\{e_{ij}\})h^k.
	\]	
	By Lemma~\ref{a24Jun26}, the centralizer of the matrix units in
	$T_x$ is precisely its centre $\CA$.
	Therefore 
	\[C_{\gA_x}(\{e_{ij}\})=\bigoplus_{k\geq 0}\CA h^k.\]
	 By \eqref{hg2}, we have $[h,u]=-2(1-u)u$ where $u=g^2$. 
	Thus the derivation $\delta$ restricts to $\CA$ and is determined by $ \delta(u)=-2(1-u)u$.  
	Consequently,
	$C_{\gA_x}(\{e_{ij}\})=\CA[h;\delta]$.
	Since $\CB$ is generated by $\CA$ and $h$, we conclude that
	$\CB=C_{\gA_x}(\{e_{ij}\})=\CA[h;\delta]$.
	The existence of the matrix units and the above identification of their centralizer then yield
	$\gA_x\cong M_2(\CB)$.
	Finally, \eqref{ZDx} gives
	$\CB\cong\mathbb K[u^{\pm1}]_{(1-u)}[h;\delta],$
	as required.
	
	(2) Set $h':=-\frac{1}{2}hu^{-1}(1-u)^{-1}\in \CB$. From $[h, u]=-2(1-u)u$, we obtain $[h',u]=1$. Thus the subalgebra of $\CB$ generated by $h'$ and $u$ is isomorphic to the first Weyl algebra $A_1$. Moreover, $\CB$ is obtained from this copy of $A_1$ by adjoining the
	inverses of $u$ and $1-u$. Thus $\CB$ is isomorphic to the localization of $A_1$ at the Ore set generated by $u$ and $1-u$; that is, $\CB\cong A_1[u^{-1},(1-u)^{-1}]$. It follows that $\CB$ is a simple algebra with centre $\mK$. By statement (1), $\gA_x \cong M_2(\CB)$, so $\gA_x$ is a simple algebra with centre $\mK$. Consequently, $Z(\gA) =\gA \cap Z(\gA_x)=\mK$. 
	
	(3) The element $x$ is regular in $\gA$, so $Q(\gA)=Q(\gA_x)$. Since $\gA_x \cong M_2(\CB)$,  
	\cite[Corollary 3.1.5]{MR} gives $Q(\gA_x) \cong M_2(Q(\CB))$. Since $\CB$ is a localization of the Noetherian domain $A_1$, it has the same classical quotient ring as $A_1$. Therefore $Q(\CB)=\Frac (A_1)=\mathbf{F}$. Consequently, $Q(\gA)=Q(\gA_x) \cong M_2(\mathbf{F})$. 
\end{proof}

\begin{corollary}
	\begin{enumerate}
		\item The prime spectrum of $\gA$ is 
		\begin{equation*}
		\Spec(\gA)=\{ \langle 0 \rangle \} \, \cup \, \{ \langle x, g\pm 1 \rangle \}\,\cup \, \{ \langle x, g\pm 1, s-\gamma \rangle \mid \gamma \in \mK \}.
		\end{equation*}
		The inclusions of prime ideals are shown in the following diagram
		\[
		\begin{tikzpicture}[scale=1]
		
		\node (pminus) at (-2.5,2) 
		{$\{\langle x,g-1,s-\gamma\rangle\mid \gamma\in\mK\}$};
		
		\node (pplus) at (2.5,2) 
		{$\{\langle x,g+1,s-\gamma\rangle\mid \gamma\in\mK\}$};
		
		\node (minus) at (-2.5,1)
		{$\langle x,g-1\rangle$};
		
		\node (plus) at (2.5,1)
		{$\langle x,g+1\rangle$};
		
		\node (zero) at (0,0)
		{$\langle0\rangle$};
		
		\draw (minus) -- (pminus);
		\draw (plus) -- (pplus);
		\draw (zero) -- (minus);
		\draw (zero) -- (plus);
		
		\end{tikzpicture}
		\]
		
		\item  The maximal spectrum of $\gA$ is 
		\begin{equation*}
		\Max(\gA)= \{ \langle x, g\pm 1, s-\gamma \rangle \mid \gamma \in \mK \}.
		\end{equation*}
		\item The primitive spectrum of $\gA$ is 
		\begin{equation*}
		\Prim(\gA)=\{ \langle 0 \rangle \} \,\cup \, \{ \langle x, g\pm 1, s-\gamma \rangle \mid \gamma \in \mK \}.
		\end{equation*}
		
		\item The nonzero prime ideals of $\gA$ are completely prime. Hence $$\Spec_c(\gA)=\Spec(\gA) \setminus \{ \langle 0 \rangle \}. $$
	\end{enumerate}
\end{corollary}

\begin{proof}
	(1) By Proposition \ref{b24Jun26}, $\gA_x \cong M_2(\CB)$ where $\CB$ is a simple domain. Hence $\gA_x$ is simple, and $\mathfrak A$ is prime. Let
	$\gp$ be a nonzero prime ideal of $\gA$. Then some power of $x$ belongs to $\gp$. Since $x$ is
	normal in $\gA$, it follows that $x\in\gp$. Thus every nonzero prime
	ideal of $\gA$ contains $x$. 
	
	Now consider the factor algebra $\bar{\gA}:=\gA/\langle x \rangle$. The defining relations give 
	\[\bar{\gA}\cong \mK[g^{\pm1},s]/\langle g^2-1\rangle.\]
	 This algebra is commutative. Since $g^2-1=(g-1)(g+1)$, every prime ideal of $\bar{\gA}$ contains either $g-1$ or
	$g+1$. Moreover, $\gA/\langle x, g \pm 1 \rangle \cong \mK[s]$. Therefore, 
	$$\Spec(\bar{\gA})=\{ \langle g\pm 1 \rangle \}\, \cup \, \{ \langle g\pm 1,  s-\gamma \rangle \mid \gamma \in \mK \}.$$
	Lifting these ideals to $\gA$ gives precisely the nonzero prime
	ideals listed in statement (1).

	(2) This follows immediately from the inclusion diagram.
	
	(3) The ideals $\langle x, g\pm 1 \rangle$ are not primitive, since $\gA/\langle x, g\pm 1 \rangle \cong \mK[s]$ has no faithful simple module. Every maximal ideal is
	primitive. Finally, the zero ideal is primitive because it is
	locally closed in $\Spec(\gA)$.
	
	(4) The algebra $\gA$ is not a domain, so $\langle 0 \rangle$ is not completely prime. For the remaining prime ideals, $\gA/\langle x,g\pm1\rangle\cong\mK[s]$ is a domain, and $\gA/\langle x,g\pm1,s-\gamma\rangle\cong\mK$. Hence all nonzero prime ideals are completely prime.
\end{proof}

\subsection{The localization $\mathfrak{U}_x$ and the classical quotient ring of $\mathfrak{U}$}

Let $\mathfrak{U}_x$ denote the localization of $\mathfrak{U}$ at the powers of the element $x$. The following theorem identifies $\mathfrak{U}_x$ as a matrix algebra, and determines the centre and classical quotient ring of $\mathfrak{U}$. 

\begin{theorem} \label{c24Jun26}  
	\begin{enumerate}
		\item Let $\mathcal R:=\mK[C]\otimes \mathcal B$. Then
		\begin{equation*}
		\mathfrak{U}_x
		\cong
		\mK[C]\otimes \gA_x
		\cong
		M_2(\mathcal R).
		\end{equation*}
		
		\item The centre of $\mathfrak{U}$ is a polynomial algebra $Z(\mathfrak{U})=\mK[C]$. 
		
		\item The classical quotient ring of $\mathfrak{U}$ is 
		\[Q(\mathfrak{U}) \cong M_2(\mathbf{K}),\]
		 where $\mathbf{K}:=\Frac(\mK[C]\otimes A_1)$. 
	\end{enumerate}
\end{theorem}
\begin{proof}
	(1) By \eqref{Uite}, we have $\mathfrak{U}=\gA[y;\s_2, \d_2]$. Localizing at the powers of $x$ yields 
	\[\mathfrak{U}_x=\gA_x[y;\s_2, \d_2].\] 
	Recall from \eqref{Cent}  that $C=\frac{1}{\l}\bigl(2yx-s+\l g^2\bigr)g^{-1}$. Multiplying by $\lambda g$ on the right, we obtain $\lambda Cg=2yx-s+\lambda g^2$. 
	Since $x$ is invertible in $\mathfrak{U}_x$, it follows that
	\begin{equation*}
	y=\frac{1}{2}(\lambda Cg+s-\lambda g^2)x^{-1}.
	\end{equation*}
	Hence $\mathfrak U_x$ is generated by $\gA_x$ and $C$. Since $C$ is central, multiplication induces a surjective homomorphism
	\begin{equation*}
	\mu\colon \mK[C] \otimes \gA_x \rightarrow \mathfrak{U}_x.
	\end{equation*}
	
	We prove that $\mu$ is injective. Rewrite the defining formula for $C$ as
	\begin{equation*}
	C=\frac{2}{\lambda}Xy+b, \qquad X:=xg^{-1}\in \gA_x, \quad b:=\frac{1}{\lambda}g^{-1}(s+\lambda g^2) \in \gA_x. 
	\end{equation*}
	Using the relation $yX=Xy+(s+x^2)g^{-1}$ and induction on $n$, we obtain
	\[
	C^n=(\frac{2}{\lambda}X)^n y^n+\text{terms of $y$-degree less than $n$}
	\]
	in the PBW expansion of the Ore extension $\mathfrak A_x[y;\sigma_2,\delta_2]$.
	Suppose that
	\[
	\sum_{i=0}^n a_iC^i=0, 
	\qquad
	a_i\in \mathfrak A_x.
	\]
	If $a_n\neq 0$, then the coefficient of $y^n$ in the PBW basis is
	$a_n\left(\frac{2}{\lambda}X\right)^n$, which is nonzero because
	$\frac{2}{\lambda}X$ is invertible. This contradiction shows that all
	$a_i=0$. Hence $\mu$ is injective, and therefore
	\[
	\mathfrak U_x\cong \Bbbk[C]\otimes \mathfrak A_x.
	\]
	 By Proposition \ref{b24Jun26}(1), $\gA_x \cong M_2(\CB)$, and consequently
	\begin{equation*}
	\mathfrak{U}_x
	\cong
	\mK[C]\otimes M_2(\mathcal B)
	\cong
	M_2(\mK[C]\otimes \mathcal B)=M_2(\mathcal R).
	\end{equation*}
	
	(2) By Proposition \ref{b24Jun26}(2), $Z(\gA_x)=\mK$. Since $\mathfrak{U}_x \cong \mK[C] \otimes \gA_x$, it follows that $Z(\mathfrak{U}_x)=\mK[C]$. Consequently, $Z(\mathfrak{U})=\mathfrak{U} \cap Z(\mathfrak{U}_x)=\mK[C]$. 
	
	(3) The element $x$ is regular in $\mathfrak{U}$, and hence localization at the powers of $x$ does not change the classical quotient ring. Thus $Q(\mathfrak{U})=Q(\mathfrak{U}_x)$. By statement (1) and \cite[Corollary 3.1.5]{MR}, 
	$Q(\mathfrak{U}_x) \cong M_2(Q(\CR)).$
	Since $\mathcal B$ is a localization of $A_1$, the domain
	$\mathcal R=\mK[C]\otimes\mathcal B$
	is a localization of $\mK[C]\otimes A_1$. Therefore, $\mathcal R$ and $\mK[C]\otimes A_1$ have the same classical quotient ring. Hence
	$Q(\mathcal R)\cong Q(\mK[C]\otimes A_1)=\mathbf K,$
	and consequently $Q(\mathfrak{U})\cong M_2(\mathbf K)$.
\end{proof}

\begin{corollary}
	The classical quotient ring of $\gB$ is 
	\[Q(\gB) \cong M_2(\mathbf{K}).\] 
\end{corollary}
\begin{proof}
	By Proposition \ref{b23Aug26}, $\mathfrak{U} =\gB_{w}$. Since $w=\lambda-X^2$ is regular in $\gB$, localization at the powers of $w$ does not change the classical quotient ring. Therefore $Q(\gB)=Q(\gB_{w})=Q(\mathfrak{U})$. The result then follows from Theorem \ref{c24Jun26}(3). 
\end{proof}

\noindent
\textbf{Data availability statement}  Data sharing is not applicable to this article as no datasets were generated or analysed during the current study.

\

\noindent
\textbf{Declarations} The author declares no conflict of interest.

\small{
	

\begin{thebibliography}{99}
		
	
		
		\bibitem{Andruskiewitsch-Angiono-Garcia Iglesias}  N. Andruskiewitsch, I. Angiono, and A. Garc\'{\i}a Iglesias, Liftings of Nichols algebras of diagonal type I. Cartan type A, {\em Int. Math. Res. Not.} (2017), 2793--2884.  
		
		
		
		\bibitem{Andruskiewitsch-Angiono-Heckenberger} N. Andruskiewitsch, I. Angiono, and I. Heckenberger, On finite GK-dimensional Nichols algebras over abelian groups, {\em Mem. Amer. Math. Soc.} {\bf 271}, 1329 (2021). 
		
		\bibitem{Andruskiewitsch-Angiono-Heckenberger-1} N. Andruskiewitsch, I. Angiono, and I. Heckenberger,  Liftings of Jordan and super Jordan planes, {\em Proc. Edinb. Math. Soc.} {\bf 61}(3) (2018), 661--672. 
		
		\bibitem{Andruskiewitsch-Angiono-Heckenberger-2} N. Andruskiewitsch, I. Angiono, and I. Heckenberger, Corrigendum: Liftings of Jordan and super Jordan planes, {\em Proc. Edinb. Math. Soc.} {\bf 65} (2022), 577--586. 
		
		\bibitem{Andruskiewitsch-Bagio-Della Flora-Flores-1} N. Andruskiewitsch, D. Bagio, S. Della Flora, and D. Fl\^{o}res, Representations of the super Jordan plane, {\em S\~{a}o Paulo J. Math. Sci.} {\bf 11} (2017), 312--325. 
		
		\bibitem{Andruskiewitsch-Bagio-Della Flora-Flores-2} N. Andruskiewitsch, D. Bagio, S. Della Flora, and D. Fl\^{o}res, On the bosonization of the super Jordan plane, {\em S\~{a}o Paulo J. Math. Sci.} {\bf 13} (2019), 1--26. 
		
		\bibitem{Andruskiewitsch-Dumas} N. Andruskiewitsch and F. Dumas, On the super Jordan plane,  {\em Int. J. Algebra Comput.} (2026), \href{https://doi.org/10.1142/S0218196726410067}{DOI: 10.1142/S0218196726410067}.
		
		\bibitem{Andruskiewitsch-Pena Pollastri} N. Andruskiewitsch and H. Pe\~{n}a Pollastri, On the double of the (restricted) super Jordan plane, {\em New York J. Math.} {\bf 28} (2022), 1596--1622. 
		
		\bibitem{Andruskiewitsch-Schneider-1} N. Andruskiewitsch and H. J. Schneider, Lifting of quantum linear spaces and pointed Hopf algebras of order $p^3$, {\em J. Algebra} {\bf 209} (2000), 1--45. 
		
		\bibitem{Andruskiewitsch-Schneider-2} N. Andruskiewitsch and H. J. Schneider, Pointed Hopf algebras, in: {\em New Directions in Hopf Algebras}, in: MSRI Series, Cambridge Univ. Press, 2002, pp. 1--68. 
		
		\bibitem{Andruskiewitsch-Schneider-3} N. Andruskiewitsch and H. J. Schneider, On the classification of finite-dimensional pointed Hopf algebras, {\em Ann. of Math.} {\bf 171} (2010), 375--417. 
		
		\bibitem{Angiono-Garcia Iglesias} I. Angiono and A. Garc\'{\i}a Iglesias, Liftings of Nichols algebras of diagonal type II: all liftings are cocycle deformations, {\em Sel. Math. New Ser. } {\bf 25}, 5(2019), 95 p. 
		
		
		
		
		
		
		\bibitem{BaseilhacEtAl2018}
		P.~Baseilhac, V.~X.~Genest,  L.~Vinet, and
		A.~Zhedanov,
		An embedding of the Bannai--Ito algebra in
		$\mathcal{U}(\mathfrak{osp}(1,2))$ and $-1$ polynomials,
		 {\em Lett.  Math. Phys.} {\bf 108} (2018), no.~7,
		1623--1634.
		
		\bibitem{BavulaGWA} V. V. Bavula, Generalized Weyl algebras and their representations, {\em St. Petersburg Math. J. } {\bf4}(1) (1993), 71--92.
		
		\bibitem{BavulaSimpleCriteria} V. V. Bavula, Filter dimension of algebras and modules, a simplicity criterion for generalized Weyl algebras, {\em Commun. Algebra} {\bf 24} (1996), 1971--1992. 
		
		\bibitem{BergeronVinet} G. Bergeron and L. Vinet, Generating functions for the $\mathfrak{osp}(1|2)$ Clebsch--Gordan coefficients, {\em J. Phys. A} {\bf 49} (2016), no. 11, 115202, 28pp. 
		
		
		
		
		
		\bibitem{Djokovic-Hochschild} D. \v{Z}. Djokovi\'{c} and G. Hochschild, Seimisimplicity of 2-graded Lie algebras. II, {\em Illinois J. Math.} {\bf 20} (1976), 134--143. 
		
		
		
		\bibitem{GenestVinetZhedanov2013}
		V.~X.~Genest, L.~Vinet, and A.~Zhedanov,
		The algebra of dual $-1$ Hahn polynomials and the Clebsch--Gordan
		problem of $sl_{-1}(2)$,
		{\em J. Math. Phys.} {\bf 54} (2013), no.~2, 023506, 13pp.
		
		\bibitem{GenestVinetZhedanov2014}
		V.~X.~Genest, L.~Vinet, and A.~Zhedanov,
		The Bannai--Ito polynomials as Racah coefficients of the
		$sl_{-1}(2)$ algebra,
		 {\em Proc. Amer. Math. Soc.} {\bf 142} (2014), no.~5, 1545--1560.
		 
		
		
		
		
		\bibitem{Jacobson} N. Jacobson, Rational methods in the theory of Lie algebras, {\em Ann. of Math.} {\bf  36} (1935), 875--881.
		
		\bibitem{LiuGongXiang} G. Liu, On Noetherian affine prime regular Hopf algebras of Gelfand--Kirillov dimension
		1, {\em Proc. Amer. Math. Soc.}{\bf 137} (2009), no. 3, 777--785. 
		
		\bibitem{Lorenz} M. Lorenz, A Tour of Representation Theory, {\em Grad. Stud. Math.} {\bf 193}, American Mathematical Society, Providence, RI, 2018. 
		
		\bibitem{Lu2026}
		T.~Lu,
		Prime spectrum and representations of the super Jordan plane,
		\newblock \emph{preprint}, arXiv:2606.31731, 2026.
		
		
		\bibitem{LuBosoJordan} T.~Lu, Prime ideals and representations of the bosonization of the super Jordan plane, {\em preprint}, arXiv:2608.31012. 
		
		\bibitem{LuJordan2026}
		T.~Lu,
		On a family of liftings of the Jordan plane,
		\newblock \emph{preprint}, arXiv:2606.31788, 2026.
		
		
		
		
		\bibitem{MR} J. C. McConnell and J. C. Robson, Noncommutative Noetherian Rings, {\em Grad. Stud. Math.} {\bf 30}, 2001.
		
		\bibitem{Pinczon1990} G. Pinczon, The enveloping algebra of the Lie superalgebra $\mathfrak{osp}(1,2)$, {\em J. Algebra} {\bf 132}(1990), no. 1, 219--242.
		
		\bibitem{Reca-Solotar} S. Reca and A. Solotar, Homological invariants relating the super Jordan plane to the Virasoro algebra, {\em J. Algebra} {\bf 507} (2018), 120--185. 
		
		\bibitem{Rowen} L. H. Rowen, Ring Theory, student edition, Academic Press, Inc., Boston, MA, 1991. 
		
		
		\bibitem{TsujimotoVinetZhedanov2011}
		S.~Tsujimoto, L.~Vinet, and A.~Zhedanov,
		From $sl_q(2)$ to a parabosonic Hopf algebra,
		\emph{SIGMA Symmetry Integrability Geom. Methods Appl.} \textbf{7}
		(2011), Paper 093, 13 pp.
		
				
		
		
	
		
	\end{thebibliography}
\end{document}